\documentclass[11pt,a4paper]{article}
\usepackage[a4paper]{geometry}
\usepackage{amssymb,latexsym,amsmath,amsfonts,amsthm}
\usepackage{graphicx}
\usepackage{epstopdf}
\usepackage{tikz}
\usepackage{pgflibraryshapes}
\usetikzlibrary{arrows,decorations.markings}
\usepackage{subfigure}
\usepackage{overpic}
\usepackage{fullpage}
\usepackage{comment,verbatim}
\usepackage{hyperref}
\usepackage[title,titletoc]{appendix}

\newtheorem{thm}{Theorem}

\theoremstyle{remark}
\newtheorem{rem}[thm]{Remark}

\numberwithin{equation}{section}

\def\Im {\mathop{\rm Im}\nolimits}
\def\arg {\mathop{\rm arg}\nolimits}
\def\Re {\mathop{\rm Re}\nolimits}
\def\Tr {\mathop{\rm Tr}\nolimits}

\numberwithin{equation}{section}

\begin{document}

\title{ Isomonodromy sets of accessory parameters
for  Heun class equations}

\author{Jun Xia\footnotemark [1], ~Shuai-Xia Xu\footnotemark [2] ~and Yu-Qiu Zhao\footnotemark [1]}

\renewcommand{\thefootnote}{\fnsymbol{footnote}}
\footnotetext [1]  { Department of Mathematics, Sun Yat-sen University, GuangZhou 510275, China.}
\footnotetext [2] {Institut Franco-Chinois de l'Energie Nucl\'{e}aire, Sun Yat-sen University,
GuangZhou 510275, China.}

\date{}
\maketitle

\begin{abstract}
In this paper, we consider the monodromy 
and, in particularly,   the isomonodromy  sets of accessory parameters for the Heun class  equations.  We  show that  the Heun class  equations can be obtained as  limits of the linear systems associated with the Painlev\'{e} equations when the Painlev\'e transcendents go  to one of the actual singular points of the linear systems.
While the isomonodromy sets of  accessory parameters for the Heun class equations are described by
the Taylor or Laurent  coefficients
of  the corresponding  Painlev\'{e} functions, or the associated tau functions, at the positions of the critical values. As an application of these results, we derive some asymptotic approximations for the isomonodromy sets of accessory parameters in the Heun class  equations, including the confluent Heun equation, the  doubly-confluent Heun equation and the reduced biconfluent Heun equation. 
\end{abstract}
\tableofcontents

\noindent
\section{Introduction and statement of results}
The Heun equation  (HE)  is the general second-order linear ODE having four regular singular points, with
the canonical form 
\cite[Eq. (31.2.1)]{OL}
\begin{equation}\label{int:HE}
\frac{\mathrm{d}^{2}w}{\mathrm{d}z^{2}}+\left(\frac{\gamma}{z}+\frac{\delta}{z-1}+\frac{\epsilon}{z-a}\right)\frac{\mathrm{d}w}{\mathrm{d}z}+\frac{\alpha\beta z-q}{z(z-1)(z-a)}w=0,
\end{equation}
where $\alpha+\beta+1=\gamma+\delta+\epsilon$.  The parameters $\alpha, \beta,\gamma,\delta, \epsilon$ determine the
 characteristic exponents of the regular singularities at $z=0,1,a,\infty$: the exponents  are   $\{0,1-\gamma\}$, $\{0,1-\delta\}$, $\{0,1-\epsilon\}$ and $\{\alpha,\beta\}$, respectively.  While the remaining parameters $a$ and $q$, known as the accessory parameters,  involve 
 global monodromy properties of 
 \eqref{int:HE}.

The same as the classical Gauss hypergeometric equation, the Heun equation 
has several confluent forms. Indeed, there are four standard confluent forms when two or more singularities merge into one or more 
irregular singularities (cf. \cite[Eqs. (31.12.1)-(31.12.4)]{OL}).\\
$(i)$ Confluent (or, singly-confluent) Heun equation (CHE):
\begin{equation}\label{int:CHE}
\frac{\mathrm{d}^{2}w}{\mathrm{d}z^{2}}+\left(1+\frac{\gamma}{z}+
\frac{\delta}{z-a}\right)\frac{\mathrm{d}w}{\mathrm{d}z}+\frac{pz-q}{z(z-a)}w=0.
\end{equation}
This equation has  two regular singularities, and an irregular singularity of  rank $1$ at infinity arising from the coalescing of two regular singularities.\\
$(ii)$ Doubly-confluent Heun equation (DHE):
\begin{equation}\label{int:DHE}
\frac{\mathrm{d}^{2}w}{\mathrm{d}z^{2}}+\left(\frac{1}{2}+\frac{\gamma}{z}-\frac{a^2}{2 z^{2}}\right)\frac{\mathrm{d}w}{\mathrm{d}z}
+\frac{pz-q}{z^{2}}w=0.
\end{equation}
This equation  has two irregular singularities of rank $1$ at zero and infinity,  each originating from the confluence of two regular singularities.\\
$(iii)$ Biconfluent Heun equation (BHE): 
\begin{equation}\label{int:BHE}\frac{\mathrm{d}^{2}w}{\mathrm{d}z^{2}}+\left(2z+2a+\frac{\gamma}{z}\right)\frac{\mathrm{d}w}{\mathrm{d}z}+\frac{pz-q}{z}w=0.
\end{equation}
This equation possesses  a regular singularity, and an irregular singularity of rank $2$ at infinity arising from the coalescing of three regular singularities.\\
$(iv)$ Triconfluent Heun equation (THE): 
\begin{equation}\label{int:THE}\frac{\mathrm{d}^{2}w}{\mathrm{d}z^{2}}+(2z^{2}+a)\frac{\mathrm{d}w}{\mathrm{d}z}+(pz-q)w=0.
\end{equation}
This equation has one  irregular singularity of rank $3$ at infinity, resulting  from the coalescing of all of the four singularities.

Modified or reduced forms
of the  confluent Heun equations
are also available. Five reduced confluent equations appear as a result of weak confluence processes. For instance, we have the reduced triconfluent Heun equation (RTHE)
\begin{equation}\label{int:RTHE}
\frac{\mathrm{d}^{2}w}{\mathrm{d}z^{2}}-\left(4z^{3}+2az+q\right)w=0,
\end{equation}
which has one irregular singularity of rank $5/2$ at infinity. There is
the reduced biconfluent Heun equation (RBHE)
\begin{equation}\label{int:RBHE}
\frac{\mathrm{d}^{2}w}{\mathrm{d}z^{2}}+\frac{2\alpha}{z}\frac{\mathrm{d}w}{\mathrm{d}z}-\left(z+a+\frac{q}{z}\right)w=0
\end{equation}
with one irregular singularity of rank $3/2$ at infinity and one regular singularity at the origin.
 Also, ten more equations are to be added to the list, if
elementary singularities are regarded as special types of singularities; see \cite{Yu3}.

The HE, together with all of its confluent  forms,  is called  the Heun class of equations; see \cite{Yu1}. This class of equations has numerous applications in theory of black holes, general relativity, polymer and chemical physics, astrophysics, molecular physics, crystalline materials and cosmology, etc.; see \cite{Ho,OL,Ron,Yu3} and the references therein.

Painlev\'{e} equations  and
Painlev\'{e} transcendents will be the main tools  in the present investigation. By Painlev\'{e} equations
we mean six   equations  denoted by PI-PVI. The six Painlev\'{e} equations were first introduced by Paul Painlev\'{e} and his coworkers 
at the turn of the twentieth century in the classification of second-order ODEs of the form
$$y_{xx}=F(x,y,y_{x})$$ with $F$ meromorphic in $x$ and rational in $y$ and $y_{x}$, whose solutions have the Painlev\'{e} property, that is, the only movable singularities of their solutions are   poles or isolated essential singularities; see  \cite{GLS,OL}. The   Painlev\'{e} equations find important applications in various fields, such as number theory, statistical mechanics, random matrix theory, orthogonal polynomials, quantum gravity and quantum field theory, nonlinear optics and fibre optics, etc.; see \cite{Cla,FIKN,For,OL} and the references therein.

In \cite{Fu1,Fu2}, Fuchs discovered a remarkable connection between the HE and the PVI. He added an extra apparent  singularity  at $z=y$ to the HE \eqref{int:HE}. The apparent singularity is presented in the equation but is absent in the solution. The position of the apparent singularity $y$ is deformed with the regular singularity $a$ in \eqref{int:HE}.
Then,  the PVI equation arose as a compatibility condition, of the deformed HE coupled with another linear equation with differentiation in the variable $a$.   Inspired by  the works of Fuchs,  Slavyanov et al. studied the  deformation of the Heun class of equations by adding an  apparent singularity.  It was shown that each Painlev\'{e} equation can be considered as the isomonodromy deformation condition for a deformed equation of the Heun class; see \cite{Yu2}-\cite{Yu4}.
In \cite{Yu1}, Slavyanov also discovered  that  the Heun class equations   can be regarded as the quantization of the classical Hamiltonian of the Painlev\'{e} equations. In this sense, there exists the following correspondence between Heun class   equations and Painlev\'{e} equations:
\begin{equation}\label{0}
\mathrm{HE} \rightarrow \mathrm{PVI},\ \mathrm{CHE} \rightarrow \mathrm{PV},\ \mathrm{BHE} \rightarrow \mathrm{PIV},\ \mathrm{DHE} \rightarrow \mathrm{PIII},\ \mathrm{THE} \rightarrow \mathrm{PII},\ \mathrm{RTHE} \rightarrow \mathrm{PI}.
\end{equation}
Here `$\rightarrow$' means that given an equation of Heun class there is a Painlev\'{e} equation corresponding to it.

 It is well-known that every Painlev\'{e} equation can be obtained as the compatibility condition of a $2\times2$ matrix linear system (also called Lax pair) \cite{Jim2}:
\begin{equation}\label{LaxPair1}
\frac{\partial\Phi(z,x)}{\partial z}=A(z,x)\Phi(z,x),
\end{equation}
and
\begin{equation}\label{LaxPair2}
\frac{\partial\Phi(z,x)}{\partial x}=B(z,x)\Phi(z,x).
\end{equation}
The first-order matrix equation \eqref{LaxPair1} is  equivalent to two second-order linear ODEs satisfied respectively by the elements of the first row and second row of $\Phi(z,x)$. More precisely, if we denote the coefficient matrix $A(z,x)$ by
$$
A(z,x):=A=\begin{pmatrix}
A_{11} & A_{12}\\
A_{21} & A_{22}
\end{pmatrix},
$$
then each element of the first row of $\Phi(z,x)$ satisfies the ODE:
\begin{equation}\label{Phi1}
\frac{\mathrm{d}^{2}\Phi_{1}}{\mathrm{d}z^{2}}-\left(\mathrm{Tr} A+\frac{A_{12}'}{A_{12}}\right)\frac{\mathrm{d}\Phi_{1}}{\mathrm{d}z}
+\left(\det A-A_{11}'+A_{11}\frac{A_{12}'}{A_{12}}\right)\Phi_{1}=0,
\end{equation}
while the elements of the second row solve a similar ODE:
\begin{equation}\label{Phi2}
\frac{\mathrm{d}^{2}\Phi_{2}}{\mathrm{d}z^{2}}-\left(\mathrm{Tr} A+\frac{A_{21}'}{A_{21}}\right)\frac{\mathrm{d}\Phi_{2}}{\mathrm{d}z}
+\left(\det A-A_{22}'+A_{22}\frac{A_{21}'}{A_{21}}\right)\Phi_{2}=0.
\end{equation}
Here, the prime indicates  the differentiation with respect to $z$.

It was observed in \cite[p.86]{FIKN} that the linear equation \eqref{Phi1} for  PVI is equivalent to HE when the solutions to PVI approach the critical values $0,1,x,\infty$, which are the actual singular points of \eqref{Phi1}.
Recently,  Dubrovin and Kapaev  \cite{DK}  studied in detail the equivalence of  HE and  the Lax pair for PVI  at the movable poles of the solutions of PVI. Moreover, they established a connection  between the accessory parameter of HE and the free parameter of the Laurent expansion of PVI.  Similar results were  also derived  earlier in \cite{LLNZ}. 
 In \cite{CC,CN1},  the equivalence of  CHE  and the Lax pair for PV at certain critical value of PV was   shown and the accessory parameter was expressed in terms of  the $\tau$-function of PV.  An application  to black holes was also addressed therein. 

In this paper, we consider the  monodromy and isomonodromy deformation of the Heun class equations.
Firstly, we  describe the monodromy of Heun class equations and consider the isomonodromy deformation by using the linear systems for the corresponding Painlev\'{e} equations.  We show  that the Heun class equations can be obtained as limits  of
the linear system \eqref{Phi1} or \eqref{Phi2}  associated with the Painlev\'e equations I-VI and XXXIV as the corresponding Painlev\'e transcendents  $y(x)$ approach  one of the actual singular points of the linear system:
\begin{equation}\label{int:HEToP1}
\mathrm{PI} \rightarrow   \mathrm{RTHE},~~~~      \mathrm{PII}  \rightarrow   \mathrm{THE}, ~~~~ \mathrm{PXXXIV} \rightarrow   \mathrm{RBHE},~~~~    \mathrm{PIII} \rightarrow \mathrm{DHE},
\end{equation}
and
\begin{equation}\label{int:HEToP2}
        \mathrm{PIV}  \rightarrow   \mathrm{BHE},~~~~    \mathrm{PV}  \rightarrow    \mathrm{CHE},~~~~   \mathrm{PVI} \rightarrow \mathrm{HE}.\end{equation}
Moreover,  the accessory parameters in these equations of Heun class  are determined by the corresponding Painlev\'{e} functions and $\tau$-functions.
Secondly,  using the the limiting procedure in \eqref{int:HEToP1}, \eqref{int:HEToP2} and the monodromy   of a specific Heun class equation, we show that there is a discrete set of pairs of accessory parameters $(a_n,q_n)$,  such that the equation of Heun class
corresponding to these parameters has the same monodromy data. Under a bijection,  the discrete set coincides with the set of parameters $(a_n, b_n)$ in the Taylor or  Laurent expansions, respectively at   the zeros or poles $a_n$, of the unique solution  to the corresponding  Painlev\'e equation with the same monodromy data  as the  Heun class equation. Finally,
using   known asymptotic expansions for the Painlev\'{e} transcendents and the associated $\tau$-functions in the literature, we  derive   some asymptotic approximations for the isomonodromy sets of  accessory parameters in the Heun class equations.

The rest of this paper is organized in the following way.  We consider the monodromy data and isomonodromy deformation of the Heun class equations RTHE-HE in Sections \ref{sec: RTHE}-\ref{sec:HE}, respectively. We describe the  isomonodromy  sets of accessary parameters by the sets of parameters  in the Taylor or Laurent expansion near the zeros or poles of the  solution of the corresponding  Painlev\'e equation with the same monodromy data  as the  Heun class equation. The main results are stated in Theorems \ref{thm:IsoSetRTHE}-\ref{thm:IsoSetHE} at the end of each section. The proofs of  Theorem  \ref{thm:IsoSetRTHE} and Theorem \ref{thm:IsoSetHE} are given in  Section \ref{sec: RTHE} and Section \ref{sec:HE}, respectively, with details included. These proofs are concerned with two  representative examples of Heun class equations,  namely the HE with four regular singularities  and the RTHE with an irregular singularity. While  Theorems  \ref{thm:IsoSetTHE}-\ref{thm:IsoSetCHE}  can be proved in the same manner and  we skip the detail. In the last section, we   derive   asymptotic approximations for
 isomonodromy sets of accessory parameters  of some Heun class equations, expressed in terms of the monodromy data. The equations of Heun class we considered in this section include the RBHE, CHE, and DHE.

\section{ Accessory parameters of RTHE }\label{sec: RTHE}

\subsection{Monodromy of  RTHE}

Consider the RTHE equation \eqref{int:RTHE} with  parameters $a$ and $q$.
There exist unique  solutions,   of the form $Y_k(z)=(y_{k1}(z), y_{k2}(z))$, which satisfy the normalized asymptotic behavior as $z\to \infty$
\begin{equation}\label{eq:RTHE-infty}Y_k(z)\sim z^{-\frac{3}{4}}(1, 1) e^{(\frac{4}{5}z^{\frac{5}{2}}+az^{\frac{1}{2}})\sigma_3}, \quad z\in \Omega_k,\end{equation}
where $ \Omega_k$,  $k=-1,\cdots,4$,  are the Stokes sectors defined by
\begin{equation}\label{eq:RTHEStokesSec}
 \Omega_k=\left\{z\in \mathbb{C}:~~\frac{\pi}{5}(2k-5)<\arg z<\frac{\pi}{5}(2k-1)\right \},
\end{equation}
 and $\sigma_3$ is one of the Pauli matrices,
 \begin{equation*}
 \sigma_1=\begin{pmatrix}
                    0 &1 \\
                   1 & 0
          \end{pmatrix},   \quad \sigma_2=
                 \begin{pmatrix}
                    0 & -i \\
                    i & 0 \end{pmatrix} \quad\mbox{and}\quad \sigma_3 =\begin{pmatrix}
                     1 & 0 \\
                    0 & -1
                    \end{pmatrix}.
                    \end{equation*}
The solutions are  related to each other by the Stokes matrices
\begin{equation}\label{eq:RTHE-Stokes1}Y_{k+1}(z)=Y_k(z)S_k, \quad  k=-1,\cdots, 3, \end{equation}
where
\begin{equation}\label{eq:RTHE-StokesM} S_{2k-1}=\begin{pmatrix}
    1 & 0 \\
    s_{2k-1} & 1
\end{pmatrix}, \quad S_{2k}=\begin{pmatrix}
    1 & s_{2k} \\
  0 & 1
\end{pmatrix}. \end{equation}
Moreover, it follows from the  asymptotic behavior in \eqref{eq:RTHE-infty} that
\begin{equation}\label{eq:RTHE-Stokes2} Y_{4}(e^{2\pi i}z)=iY_{-1}(z)\sigma_1. \end{equation}
Using \eqref{eq:RTHE-Stokes1}-\eqref{eq:RTHE-Stokes2}, we find  the cyclic  condition
 \begin{equation}\label{eq:RTHE-CyclicCon}S_{-1}\cdots S_3=i\sigma_1. \end{equation}
 The restriction \eqref{eq:RTHE-CyclicCon} contributes four scalar algebraic equations, three among them are independent.
 Accordingly, 
 the Stokes multipliers $s_k$ fulfill 
 \begin{equation}\label{eq:RTHE-StokesEq}
s_{k}=i(1+s_{k+2}s_{k+3})
\end{equation}for all integer $k$, with  $s_{k+5}=s_k$.
The monodromy data is constituted by the set of Stokes multipliers, which is a 2-dimensional surface in $\mathbb{C}^5$ described
by \eqref{eq:RTHE-StokesEq}
\begin{equation}\label{eq:RTHEMonDat}\{(s_{1}, s_{2},s_{3},s_{4},s_{5})\in \mathbb{C}^5: s_{1},\dots s_{5}~  \mbox{satisfy}~  \eqref{eq:RTHE-StokesEq} \}.\end{equation}

\subsection{Isomonodromy deformation and PI equation}

To study the isomonodromy deformation of the  RTHE equation, it is  convenient to consider a $2\times 2$ matrix  system, which has one irregular singular point of rank 5/2 at infinity. In \cite{FIKN,Jim2}, the isomonodromy deformation of such a matrix  system
 has been considered. We review it briefly below in this section.

Consider the following  Lax pair for PI parameterized in the way as \cite[(C.2)]{Jim2}
\begin{equation}\label{LaxPair-PI}
\left\{\begin{aligned}
&\frac{\partial\Phi(z,x)}{\partial z}=A(z,x)\Phi(z,x),\\
&\frac{\partial\Phi(z,x)}{\partial x}=B(z,x)\Phi(z,x),
\end{aligned}\right.
\end{equation}
where
$$
A(z,x)=
\begin{pmatrix}
-v & z^{2}+yz+y^{2}+\frac{x}{2}\\
4(z-y) & v
\end{pmatrix},\ \ \ \ \
B(z,x)=
\begin{pmatrix}
0 & \frac{z}{2}+y\\
2 & 0
\end{pmatrix}.
$$
The compatibility condition of the Lax pair reads
\begin{equation}\label{CCP1}\left\{
\begin{aligned}
&\frac{\mathrm{d}y}{\mathrm{d}x}=v,\\
&\frac{\mathrm{d}v}{\mathrm{d}x}=6y^{2}+x,
\end{aligned}\right.
\end{equation}
which leads to the PI equation
$$
\frac{\mathrm{d}^{2}y}{\mathrm{d}x^{2}}=6y^{2}+x.
$$
The $\tau$-function associated with PI is defined by (see \cite[(C.7)]{Jim2})
\begin{equation}\label{TauP1}
\sigma(x)=\frac{\mathrm{d}}{\mathrm{d}x}\log \tau(x)=\frac{1}{2}v^{2}-2y^{3}-xy.
\end{equation}
Then, we have (see \cite[(C.9)]{Jim2})
\begin{equation}\label{SigmaFormP1}
\left(\frac{\mathrm{d}^2\sigma}{\mathrm{d}x^2}\right)^2+4\left(\frac{\mathrm{d}\sigma}{\mathrm{d}x}\right)^2+2x \frac{\mathrm{d}\sigma}{\mathrm{d}x}-2\sigma=0.
\end{equation}

There exist unique solutions $\Phi_k(z,x)$ of the equation \eqref{LaxPair-PI}, which satisfy the  normalized asymptotic behavior as $z\to \infty$
\begin{equation}\label{eq:P1-infty}
\Phi_k(z,x)=\left(\frac{z}{4}\right)^{\frac{1}{4}\sigma_3}
\frac{\sigma_3+\sigma_1}{ i\sqrt{2}}\left (I+O\left(z^{-\frac{1}{2}}\right ) \right )e^{\left(\frac{4}{5}z^{\frac{5}{2}}+xz^{\frac{1}{2}}\right )\sigma_3}, \quad z\in \Omega_k, \end{equation}
where $ \Omega_k$ for $k=-1,\cdots,4$ are the Stokes sectors defined by
\eqref{eq:RTHEStokesSec}. From \eqref{eq:P1-infty}, we see that the solutions are  related to each other by the Stokes matrices
\begin{equation}\label{eq:P1-Stokes1}\Phi_{k+1}(z)=\Phi_k(z)\hat{S}_k, \quad  k=-1,\cdots,3, \end{equation}
where
\begin{equation}\label{eq:P1-StokesM} \hat{S}_{2k-1}=\begin{pmatrix}
    1 & 0 \\
    \hat{s}_{2k-1} & 1
\end{pmatrix}, \quad \hat{S}_{2k}=\begin{pmatrix}
    1 & \hat{s}_{2k} \\
  0 & 1
\end{pmatrix}. \end{equation}
The behavior \eqref{eq:P1-infty}  also gives
\begin{equation}\label{eq:P1-Stokes2} \Phi_{4}(e^{2\pi i}z)=i\Phi_{-1}(z)\sigma_1. \end{equation}
From \eqref{eq:P1-Stokes1}-\eqref{eq:P1-Stokes2}, we find the cyclic  condition
 \begin{equation}\label{eq:P1-CyclicCon}\hat{S}_{-1}\hat{S}_0\cdots\hat{S}_3= i\sigma_1.  \end{equation}
This condition  implies the following algebraic equations
 \begin{equation}\label{eq:P1-StokesEq}
 \hat{s}_{k}=i(1+\hat{s}_{k+2}\hat{s}_{k+3}),
\end{equation}
with $k\in \mathbb{Z}$ and $\hat{s}_{k+5}=\hat{s}_{k}$.
Thus, the Stokes multipliers are described by the 2-dimensional surface \eqref{eq:RTHEMonDat}  in $\mathbb{C}^5$ with the coordinates  $(\hat{s}_{k})_{1\leq k\leq 5}$.

\subsection{Reduction of the  linear system  for PI to RTHE} \label{sec: PIRTHE}

In this subsection, we show that  the RBHE can  be obtained  as a limit of the  linear system   \eqref{Phi2} associated with   the Lax pair  for the Painlev\'{e} I equation when $x$ tends to one of the poles of the Painlev\'e I transcendents.

Substituting the enties of $A(z,x)$ into \eqref{Phi2}, we obtain
\begin{equation}\label{Phi2PI}
\frac{\mathrm{d}^{2}\Phi_{2}}{\mathrm{d}z^{2}}-\frac{1}{z-y}\frac{\mathrm{d}\Phi_{2}}{\mathrm{d}z}-Q(z,x)\Phi_{2}=0,
\end{equation}
where
\begin{align*}
Q(z,x)&=4z^{3}+2xz+v^{2}-4y^{3}-2xy-\frac{v}{z-y}\nonumber\\
&=4z^{3}+2xz+2\frac{\mathrm{d}}{\mathrm{d}x}\log \tau(x)-\frac{v}{z-y}.
\end{align*}

To obtain RTHE, we need to eliminate the singularity $z=y$ in \eqref{Phi2PI}. This can be realized by considering the poles of $y$.
It is known that $y$ admits the following Laurent expansion near a pole  (see \cite[Eq. (1.1)]{GLS} ):
\begin{equation}\label{PIExpand}
y(x)=\frac{1}{(x-a)^{2}}-\frac{a}{10}(x-a)^{2}-\frac{1}{6}(x-a)^{3}+b(x-a)^{4}+O((x-a)^{5}),
\end{equation}
where $b$ is an arbitrary parameter.

Using the first equation of \eqref{CCP1}  to replace $v$ by $y'$  and  substituting \eqref{PIExpand} into the expression of $Q$ in \eqref{Phi2PI},  we obtain the RTHE as  $x\rightarrow a$
\begin{equation}\label{RTHE}
\frac{\mathrm{d}^{2}\Phi_{2}}{\mathrm{d}z^{2}}-\left(4z^{3}+2az+q\right)\Phi_{2}=0,
\end{equation}
with the accessory parameter $q$ given by
\begin{equation}\label{PIq}
q=\lim_{x\rightarrow a}2\left(\frac{\mathrm{d}}{\mathrm{d}x}\log \tau(x)-\frac{1}{x-a}\right)=-28b.
\end{equation}


\subsection{Isomonodromy set of  accessory parameters of RTHE}

As shown in Section \ref{sec: PIRTHE}, the isomonodromy deformation of \eqref{LaxPair-PI}  is described by the PI equation. While 
 the RTHE can be obtained  as a limit of the second row of  the isomonodromy family of  $\Phi(z,x)$,  when $x$ tends to one of the pole of  the  Painlev\'e I transcendent.
The accessory parameters of RTHE are then expressed in terms of the
parameters  in the  Laurent expansion  near the poles  of the  PI transcendents. As a consequence, we obtain a family of accessory parameters such that the corresponding RTHE shares the same monodromy data, as stated in the following theorem.

\begin{thm}\label{thm:IsoSetRTHE}
There is a discrete set of pairs of accessory parameters $(a_n,q_n)$  such that  the RTHE  \eqref{int:RTHE}
corresponding to these parameters has the same monodromy data  \eqref{eq:RTHEMonDat}  as the original
one with the parameters $a$ and $q$. This set coincides, via the transformation $q_n=-28b_n$, with the set of parameters $(a_n, b_n)$ in the  Laurent expansion \eqref{PIExpand} near the poles $a_n$, of the unique solution  of PI corresponding to  the same monodromy data \eqref{eq:RTHEMonDat} as the RTHE.
\end{thm}

\begin{proof}
Consider the RTHE  \eqref{int:RTHE}  with given accessory parameters $a$ and $q$.
There exist unique  solutions of  \eqref{int:RTHE} which satisfy the normalized asymptotic behavior \eqref{eq:RTHE-infty} as $z\to \infty$.  The solutions are related by the Stokes matrices $S_k$ as given in
\eqref{eq:RTHE-Stokes1} and \eqref{eq:RTHE-StokesM}, while the independent Stokes multipliers constitute the
monodromy data \eqref{eq:RTHEMonDat} for the RTHE.
Moreover, there exist unique solutions $\Phi_k(z,x)$ of the system \eqref{LaxPair-PI} with the normalized behavior \eqref{eq:P1-infty}  at infinity, where the $x$-dependence of $\Phi_k(z,x)$ is described by the unique solution of the PI equation  having the Laurent expansion \eqref{PIExpand}. Here we let $a$ be a pole of the PI solution and  choose the parameter $b$ according to \eqref{PIq}, namely
$b=-\frac{q}{28}$.  The solutions $\Phi_{k}(z, x)$ are related by the Stokes matrices $\hat{S}_k$ defined in
\eqref{eq:P1-Stokes1} and \eqref{eq:P1-StokesM},  which are independent of $x$ and have the same triangular-matrix forms as the Stokes matrices $S_k$ for the RTHE.

 From  the limiting procedure demonstrated  in Section \ref{sec: PIRTHE},
 we obtain the  RTHE equation with accessory parameters $a$  and $q$, as a  limit of the second row of the isomonodromy family  $\Phi_k(z;x)$ when $x\to a$.  It follows from the $x$-independence of the monodromy data that the Stokes matrices $S_k=\hat{S}_k$. Thus, we have shown that any given accessory parameters $(a,q)$ is related, via $q=-28b$,  to the pole parameters $(a,b)$ of the unique solution of PI corresponding to the same monodromy data \eqref{eq:RTHEMonDat}. It is known that each Painlev\'e tanscendents has infinity many poles which are  discrete in the complex plane. Therefore the set of pairs of accessory parameters of RTHE sharing the same monodromy data are also discrete.  Thus, we complete the proof of Theorem \ref{thm:IsoSetRTHE}.
\end{proof}

\section{Accessory parameters of THE}\label{sec:THE}

\subsection{Monodromy of THE}

Consider the THE \eqref{int:THE} with  the accessory  parameters $a$ and $q$ and the fixed parameter $p$.
There exist unique  solutions of \eqref{int:THE} 
with normalized asymptotic behavior as $z\to \infty$
\begin{equation}\label{eq:THE-infty}Y_k(z)\sim e^{-(\frac{1}{3}z^{3}+\frac{a}{2}z)}z^{-1}(1,1) e^{(\frac{1}{3}z^{3}+\frac{a}{2}z)\sigma_3}z^{(\mu-\frac{1}{2})\sigma_3}, \quad z\in \Omega_k,\end{equation}
where the Stokes sectors are defined by
\begin{equation}\label{eq:THEStokesSec}
 \Omega_k=\left\{z\in \mathbb{C}:~~ \frac{\pi}{6}\left (2k-3\right )<\arg z<\frac{\pi}{6}\left(2k+1\right)\right \},
\end{equation}
for $k=1,\cdots,7$.  Here the parameter $\mu$ is given by $\mu=(3-p)/2$. The solutions are  related by the Stokes matrices
\begin{equation}\label{eq:THE-Stokes1}Y_{k+1}(z)=Y_k(z)S_k, \quad  k=1,\dots,6, \end{equation}
where
\begin{equation}\label{eq:THE-StokesM}S_{2k-1}=\begin{pmatrix}
    1 &  s_{2k-1} \\
   0 & 1
\end{pmatrix}, \quad S_{2k}=\begin{pmatrix}
    1 &0 \\
  s_{2k} & 1
\end{pmatrix}. \end{equation}
The behavior \eqref{eq:THE-infty}  also gives
\begin{equation}\label{eq:THE-Stokes2} Y_{7}(e^{2\pi i}z)=Y_{1}(z)e^{\pi i(2\mu-1)\sigma_3}. \end{equation}
We define the monodromy data of THE by
\begin{equation}\label{eq:THEMonDat}\left\{S_k\right\}_{k=1}^6 .\end{equation}
From \eqref{eq:THE-Stokes1}-\eqref{eq:THE-Stokes2}, we see that the Stokes matrices fulfill the  cyclic  condition
 \begin{equation}\label{eq:THE-CyclicCon}S_{1}S_{2}\cdots S_6=e^{\pi i(2\mu-1)\sigma_3}.  \end{equation}
 Thus,  the Stokes multipliers $s_k$, $ k=1,\cdots,6$ are constrained by  the  following algebraic equations
 \begin{equation}\label{eq:THE-CyclicConEq}
 \left\{\begin{aligned}
&1+s_1s_2+e^{2\pi i\mu}(1+s_4s_5)=0,\\
&1+s_2s_3+e^{-2\pi i\mu}(1+s_5s_6)=0,\\
&s_1+s_3+s_1s_2s_3-e^{2\pi i\mu}s_5=0. 
\end{aligned}\right.
  \end{equation}
These equations describe a 3-dimensional  surface in $\mathbb{C}^6$ with the coordinates $(s_k)_{1\leq k\leq 6}.$

\subsection{Isomonodromy deformation and PII equation} \label{IsoDefPII}

To study the isomonodromy deformation of THE, it is  convenient to consider a $2\times 2$ matrix  system, which has one irregular singular point of rank 3 at infinity. In \cite{Jim2}, the isomonodromy deformation of the matrix  system with such an irregular singular point has been considered.

Recall the following Lax pair for PII with the parameterization given in \cite[ (C.10)-(C.11)]{Jim2}:
\begin{equation}\label{LaxPair-PII}
\left\{\begin{aligned}
&\frac{\partial\Phi(z,x)}{\partial z}=A(z,x)\Phi(z,x),\\
&\frac{\partial\Phi(z,x)}{\partial x}=B(z,x)\Phi(z,x),
\end{aligned}\right.
\end{equation}
where
$$
A(z,x)=
\begin{pmatrix}
z^{2}+v+\frac{x}{2} & u(z-y)\\
-\frac{2}{u}\big(vz+vy+\frac{1}{2}-\mu\big) & -z^{2}-v-\frac{x}{2}
\end{pmatrix},\ \ \ \ \
B(z,x)=
\begin{pmatrix}
\frac{1}{2}z & \frac{1}{2}u\\
-\frac{v}{u} & -\frac{1}{2}z
\end{pmatrix}.
$$
The compatibility condition of the above Lax pair gives
\begin{equation}\label{CCP2}
\left\{
\begin{aligned}
&\frac{\mathrm{d}y}{\mathrm{d}x}=y^{2}+v+\frac{x}{2}, \\
&\frac{\mathrm{d}v}{\mathrm{d}x}=-2yv-\frac{1}{2}+\mu,\\
&\frac{\mathrm{d}u}{\mathrm{d}x}=-uy,
\end{aligned}\right.
\end{equation}
which yields the PII equation
\begin{equation}\label{PII}
\frac{\mathrm{d}^{2}y}{\mathrm{d}x^{2}}=2y^{3}+xy+\mu.
\end{equation}
The PII $\tau$-function is defined by (see \cite[(C.15)]{Jim2})
\begin{equation}\label{TauP2}
\frac{\mathrm{d}}{\mathrm{d}x}\log\tau(x)=\frac{1}{2}v^{2}+\left(y^{2}+\frac{x}{2}\right)v+\left(\frac{1}{2}-\mu\right)y.\end{equation}

There exist unique solutions $\Phi_k(z,x)$ of the equation \eqref{LaxPair-PII} specified by the  normalized asymptotic behavior as $z\to \infty$
\begin{equation}\label{eq:P2-infty}
\Phi_k(z,x)=(I+O(1/z) )e^{(\frac{1}{3}z^{3}+\frac{1}{2}xz)\sigma_3}z^{(\mu-\frac{1}{2})\sigma_3}, \quad z\in \Omega_k, \end{equation}
where $\Omega_k$ for $k=1,\cdots,7$ are the Stokes sectors defined in \eqref{eq:THEStokesSec}.
From \eqref{eq:P2-infty}, we see that the solutions are  related to each other by the Stokes matrices
\begin{equation}\label{eq:P2-Stokes1}\Phi_{k+1}(z)=\Phi_k(z)\hat{S}_k, \quad  k=1,\cdots,6, \end{equation}
where
\begin{equation}\label{eq:P2-StokesM} \hat{S}_{2k-1}=\begin{pmatrix}
    1 &  \hat{s}_{2k-1} \\
   0 & 1
\end{pmatrix}, \quad \hat{S}_{2k}=\begin{pmatrix}
    1 &0 \\
   \hat{s}_{2k} & 1
\end{pmatrix}. \end{equation}
We define the monodromy data of the system \eqref{LaxPair-PII} by
\begin{equation}\label{eq:P2MonDat}\{\hat{S}_k\}_{k=1,\dots, 6}.\end{equation}
The behavior \eqref{eq:P2-infty}  also implies
\begin{equation}\label{eq:P2-Stokes2} \Phi_{7}(e^{2\pi i}z)=\Phi_{1}(z)e^{\pi i(2\mu-1)\sigma_3}. \end{equation}
From \eqref{eq:P2-Stokes1}-\eqref{eq:P2-Stokes2}, we find the cyclic  condition
 \begin{equation}\label{eq:P2-CyclicCon}\hat{S}_{1}\hat{S}_{2}\cdots\hat{S}_6=e^{\pi i(2\mu-1)\sigma_3}.  \end{equation}
 The equation gives us  three independent equations of the Stokes multipliers $\hat{s}_k$, $  k=1,\cdots,6$ as given in \eqref{eq:THE-CyclicConEq}; see also \cite[Proposition 5.3]{FIKN}.

\subsection{Reduction of the  linear system   for PII to  THE}\label{sec:PIItoTHE}
In this subsection, we derive THE from the linear system \eqref{Phi1} for PII at each pole of  the solutions of PII.

Substituting  $A(z,x)$ into  \eqref{Phi1} gives
\begin{equation}\label{Phi1P2}
\frac{\mathrm{d}^{2}\Phi_{1}}{\mathrm{d}z^{2}}-\frac{1}{z-y}\frac{\mathrm{d}\Phi_{1}}{\mathrm{d}z}-Q(z,x)\Phi_{1}=0,
\end{equation}
where
\begin{align*}
Q(z,x)&=z^{4}+xz^{2}+(1+2\mu)z+\frac{x^{2}}{4}+v^{2}+\left(2y^{2}+x\right)v+\left(1-2\mu\right)y-\frac{z^{2}+\frac{x}{2}+v}{z-y}\\
&=z^{4}+xz^{2}+(1+2\mu)z+\frac{x^{2}}{4}+2\frac{\mathrm{d}}{\mathrm{d}x}\log \tau(x)-\frac{z^{2}+\frac{x}{2}+v}{z-y}.
\end{align*}

To eliminate the singularity $z=y$ in \eqref{Phi1P2}, we consider the poles of $y$.
Let $x=a$ be a pole of $y$, then $y$ possesses the following Laurent expansion (see \cite[(17.1)]{GLS}):
\begin{equation}\label{yPIIExpand}
y(x)=\frac{\varepsilon}{x-a}-\frac{\varepsilon a}{6}(x-a)-\frac{\mu+\varepsilon}{4}(x-a)^{2}+b(x-a)^{3}+O((x-a)^{4}),
\end{equation}
where $\varepsilon=\varepsilon_{\pm}=\pm1$ and $b\in \mathbb{C}$ is arbitrary.
From the first equation of the compatibility condition in \eqref{CCP2} and \eqref{yPIIExpand}, we obtain
\begin{equation}\label{vP2Expand}
v(x)=\left\{\begin{aligned}
  &-\frac{2}{(x-a)^{2}}+O(1), &  \varepsilon=\varepsilon_{+},\\
  & O(x-a),  \  &  \varepsilon=\varepsilon_{-}.
\end{aligned} \right.
\end{equation}

Hence, taking $x\rightarrow a$ in \eqref{Phi1P2} and then making the transformation
$$
\Phi_{1}=w\exp\left(\frac{1}{3}z^{3}+\frac{a}{2}z\right),
$$
we obtain the THEs
\begin{equation}\label{THE1}
\left\{\begin{aligned}
&\frac{\mathrm{d}^{2}w}{\mathrm{d}z^{2}}+(2z^{2}+a)\frac{\mathrm{d}w}{\mathrm{d}z}+(pz-q)w=0,\\
&p=3-2\mu,\\
&q=\lim_{x\rightarrow a}2\left(\frac{\mathrm{d}}{\mathrm{d}x}\log\tau(x)-\frac{1}{x-a}\right)
=-\frac{a^{2}}{18}-10b,
\end{aligned}\right.
\end{equation}
for $\varepsilon=\varepsilon_{+}=1$,
\begin{equation}\label{THE2}
\left\{
\begin{aligned}
&\frac{\mathrm{d}^{2}w}{\mathrm{d}z^{2}}+(2z^{2}+a)\frac{\mathrm{d}w}{\mathrm{d}z}+(pz-q)w=0,\\
&p=1-2\mu,\\
&q=\lim_{x\rightarrow a}2\frac{\mathrm{d}}{\mathrm{d}x}\log\tau(x)=-\frac{a^{2}}{18}+10b,
\end{aligned}
\right.
\end{equation}
for $\varepsilon=\varepsilon_{-}=-1$.

%

%
\subsection{Isomonodromy set of  accessory parameters of THE}

Thus, we  see that  THE can be obtained  as a limit of the first row of  the isomonodromy family of  $\Phi(z,x)$,  when  $x\to a$, with $a$ being the pole of  the  Painlev\'e II transcendent  corresponding to the same monodromy data of THE.  While, the accessory parameters $(a,q)$ are expressed in terms of the parameters in the  Laurent expansion  of the  PII transcendent. Consequently, we have the following description of the isomonodromy set of  accessory parameters of THE.

\begin{thm}\label{thm:IsoSetTHE}
There is a discrete set of pairs of accessory parameters $(a_n,q_n)$  such that the THE  \eqref{int:THE},  corresponding to these parameters has the same monodromy data \eqref{eq:THEMonDat} and \eqref{eq:THE-CyclicConEq} as the original equation with the parameters $a$ and $q$. Under the bijection given in the third equation of  \eqref{THE1},  this set coincides with the set of parameters $(a_n, b_n)$ in the  Laurent expansion  (cf. \eqref{yPIIExpand}) near the poles $a_n$ of the unique solution  of PII equation \eqref{PII} with the parameter  $ \mu= (3-p)/2$ and the same monodromy data \eqref{eq:THEMonDat} and \eqref{eq:THE-CyclicConEq}  as the THE.
\end{thm}

\section{ Accessory parameters of RBHE }\label{sec:RBHE}

\subsection{Monodromy  of RBHE}

Consider the RBHE \eqref{int:RBHE} with  the fixed parameter $2\alpha$ and the accessory parameters $a$ and $q$.
There exist two linear independent solutions $Y_k=(y_{k1},y_{k2})$ of \eqref{int:RBHE}   normalizing  the 
asymptotic behavior as $z\to \infty$
\begin{equation}\label{eq:RBHE-infty}Y_k(z)\sim z^{-(\alpha+\frac{1}{4})}(1,i) e^{-(\frac{2}{3}z^{\frac{3}{2}}+az^{\frac{1}{2}})\sigma_3}, \quad z\in \Omega_k, \end{equation}
where  $\Omega_k$ for $k=-1,\cdots,2$   are the Stokes sectors
\begin{equation}\label{eq:P34StokesSec}
 \Omega_k=\left\{z\in \mathbb{C}:~~ \frac{\pi}{3}(2k-3)<\arg z<\frac{\pi}{3}(2k+1) \right \}.
\end{equation}
 The solutions are  related to each other by the Stokes matrices
\begin{equation}\label{eq:RBHE-Stokes1}Y_{k+1}(z)=Y_k(z)S_k, \quad  k=-1,0,1, \end{equation}
where
\begin{equation}\label{eq:RBHE-StokesM} S_{\pm 1}=\begin{pmatrix}
    1 & 0 \\
   s_{\pm 1} & 1
\end{pmatrix}, \quad S_{0}=\begin{pmatrix}
    1 & s_{0} \\
  0 & 1
\end{pmatrix}. \end{equation}
The behavior \eqref{eq:RBHE-infty} 
also gives
\begin{equation}\label{eq:RBHE-Stokes2} Y_{2}(e^{2\pi i}z)=e^{-2\pi i\alpha}Y_{-1}(z)i\sigma_2. \end{equation}
Moreover,
from
\eqref{int:RBHE} we see that the characteristic exponents   at $z=0$ are  $\{0,1-2\alpha\}$.
Hence,  as $z\to 0$, we have the asymptotic behavior
\begin{equation}\label{eq:RBHE-asy-O} Y_2(z) \sim z^{-\alpha}(1,z)z^{\alpha\sigma_3}E \end{equation}
for $2\alpha\not\in \mathbb{Z}$,~
 with $E$ being  some invertible constant matrix.
Using \eqref{eq:RBHE-Stokes1}-\eqref{eq:RBHE-asy-O}, we obtain the same cyclic  condition as that for the Lax pair of the PXXXIV equation
\begin{equation}\label{eq:RBHE-CyclicCon}S_{-1}S_0S_1= i\sigma_2E^{-1}e^{-2\pi i\alpha\sigma_3}E.  \end{equation}
This condition  implies the following algebraic equation
 \begin{equation}\label{eq:RBHE-StokesEq}
 s_{-1}s_{0}s_{1}+s_{-1}+s_{1}-s_{0}=-2\cos(2\pi \alpha).
\end{equation}
Thus,  the monodromy data for  RBHE \eqref{int:RBHE} 
are described by the  2-dimensional complex surface
\begin{equation}\label{eq:RBHEMonDat}\left\{(s_{-1}, s_{0}, s_{1})\in \mathbb{C}^3:~~s_{-1}s_{0}s_{1}+s_{-1}+s_{1}-s_{0}=-2\cos(2\pi \alpha) \right \}.\end{equation}

\subsection{Isomonodromy  deformation  and PXXXIV equation}

To study the isomonodromy deformation of the  RBHE \eqref{int:RBHE}, we regard $z^{\alpha}Y_k(z)$ as the first row of a  $2\times 2$ matrix  system  $\Phi_k(z,x)$ when $x\to a$. Then, we look for the system $\Phi_k(z,x)$ such that  $\Phi_k(z,x)$ satisfy the same jump relations \eqref{eq:RBHE-Stokes1}-\eqref{eq:RBHE-Stokes2} and the following asymptotic behaviors
\begin{equation}\label{eq:P34-infty}
\Phi_k(z,x)=z^{-\frac{1}{4}\sigma_3}\frac{I+i\sigma_1}{\sqrt{2}}
\left(I+O\left(z^{-\frac{1}{2}}\right) \right )e^{-\left(\frac{2}{3}z^{\frac{3}{2}}+xz^{\frac{1}{2}}\right)\sigma_3}, \quad z\in \Omega_k, \end{equation}
as $z\to \infty$, where  $\Omega_k$, $k=-1,\cdots, 2$,  are the Stokes sectors defined in
\eqref{eq:P34StokesSec},
and
\begin{equation}\label{eq:P34-zero}\Phi_2(z,x) =\Phi^{(0)}(x)\left(I+O(z)\right)z^{\alpha\sigma_3}E,~~z\to 0,
\end{equation}
where the parameter $2\alpha\not\in \mathbb{Z}$ and $\Phi^{(0)}(x)$ is a matrix independent of $z$.
The behavior of $\Phi_k(z,x) $ as $z\to 0$ are then determined by \eqref{eq:P34-infty} and the jump relations.

It can be shown that $\Phi_k(z,x)$ satisfies the following Lax pair with rational coefficients  (see \cite[Chapter 5]{FIKN},  \cite[Lemma 3.2]{ik} and \cite{WXZ})
\begin{equation}\label{LaxPairP34}
\left\{\begin{aligned}
&\frac{\partial\Phi(z,x)}{\partial z}=A(z,x)\Phi(z,x),\\
&\frac{\partial\Phi(z,x)}{\partial x}=B(z,x)\Phi(z,x),
\end{aligned}\right.
\end{equation}
where
$$
A(z,x)=
\begin{pmatrix}
\frac{y'}{2z} & i-i\frac{y}{z}\\
-iz-i(y+x)-i\frac{(y')^2-(2\alpha)^2}{4yz} & -\frac{y'}{2z}
\end{pmatrix},\ \ \
B(z,x)=
\begin{pmatrix}
0 & i\\
-iz-2i(y+\frac{x}{2}) & 0
\end{pmatrix}.
$$
The Lax pair is related by a gauge transformation to the  Lax pair for  PII found by Flaschka and Newell; see for instance  \cite[Chapter 5]{FIKN} and \cite[Equation (3.19)]{ik}.

The compatibility condition of the above Lax pair is described by the PXXXIV equation
\begin{equation}\label{P34}
\frac{\mathrm{d}^{2}y}{\mathrm{d}x^{2}}=\frac{1}{2y}\left(\frac{\mathrm{d}y}{\mathrm{d}x}\right)^{2}+4y^{2}+2xy-\frac{(2\alpha)^{2}}{2y}.
\end{equation}
Let
$$v(x)=\frac{y'(x)-2\alpha}{2y(x)},
$$
then $-2^{-1/3}v(-2^{-1/3}x)$ satisfies the PII equation   \eqref{PII} with $\mu=-(2\alpha+\frac{1}{2})$.
The associated $\tau$-function can be defined as
\begin{equation}\label{TauP34}
\frac{\mathrm{d}}{\mathrm{d}x}\log\tau(x)=-y(x)^{2}+
\left(v(x)^{2}-x\right )y(x)+2\alpha v(x),
\end{equation}
which is related simply to the standard Hamiltonian for the PII equation; cf. \cite[ Equation (32.6.9)]{OL}.

Thus,
we  see that the solutions $\Phi_k(z,x)$ to the Lax pair   \eqref{LaxPairP34} subject to the boundary conditions \eqref{eq:P34-infty}-\eqref{eq:P34-zero} share 
the same monodromy data as RBHE (independent of $x$) given in   \eqref{eq:RBHEMonDat}. While  the isomonodromy deformation of $\Phi_k(z,x)$  with respect to $x$ are described by the PXXXIV equation.

In the next subsection, we show that the RBHE can  actually be obtained  as a limit of the  linear system \eqref{Phi1},  associated with  the Lax pair  for the Painlev\'{e} XXXIV equation, when the Painlev\'e XXXIV transcendent   $y(x)$ tends to   zero or infinity.

\subsection{Reduction of the  linear system for PXXXIV to RBHE}
Substituting the elements of $A(z,x)$ into \eqref{Phi1} gives
\begin{equation}\label{Phi1P34}
\frac{\mathrm{d}^{2}\Phi_{1}}{\mathrm{d}z^{2}}-\left(\frac{1}{z-y}-\frac{1}{z}\right)\frac{\mathrm{d}\Phi_{1}}{\mathrm{d}z}-P(z,x)\Phi_{1}=0,
\end{equation}
where
\begin{align*}
P(z,x)&=z+x+\frac{((y')^2-4\alpha^2)/(4y)-y^2-xy}{z}+\frac{\alpha^2}{z^{2}}-\frac{y'}{2z(z-y)}\\
&=z+x+\frac{1}{z}\left(\frac{\mathrm{d}}{\mathrm{d}x}\log\tau(x)\right)+\frac{\alpha^2}{z^{2}}-\frac{y'}{2z(z-y)}.
\end{align*}
To eliminate the singularity $z=y$, we consider the poles and zeros of $y$. It is readily seen that
$y$ solving \eqref{P34} 
possesses the following Laurent expansion at a pole
\begin{equation}\label{P34ExpandAtpole}
y(x)=\frac{1}{(x-a)^{2}}-\frac{a}{3}-\frac{1}{2}(x-a)+b(x-a)^{2}+O\left((x-a)^{3}\right ),
\end{equation}
and the Taylor expansion at a zero
\begin{equation}\label{P34ExpandAtzero}
y(x)=\varepsilon(x-a)+b(x-a)^{2}+\frac{2a\varepsilon}{3}(x-a)^{3}+O\left((x-a)^{4}\right),
\end{equation}
where $\varepsilon=\varepsilon_{\pm}=\pm2\alpha$ and $b$ is arbitrary.

Taking $x\rightarrow a$ in \eqref{Phi1P34} and making the transformation
\begin{equation}\label{P34PhiTransf}\Phi_{1}=z^{\alpha}w,\end{equation}
we see that $w(z)$ satisfies the following RBHEs
\begin{equation}\label{RBHE1}
\left\{\begin{aligned}
&\frac{\mathrm{d}^{2}w}{\mathrm{d}z^{2}}+\frac{2\alpha+1}{z}\frac{\mathrm{d}w}{\mathrm{d}z}-\left(z+a+\frac{q}{z}\right)w=0,\\
&q=\lim_{x\rightarrow a}\left(\frac{\mathrm{d}}{\mathrm{d}x}\log\tau(x)-\frac{1}{x-a}\right)=\frac{a^2}{3}-5b,
\end{aligned}\right.
\end{equation}
when $a$ is a pole of $y(x)$,  and
\begin{equation}\label{RBHE2}
\left\{\begin{aligned}
&\frac{\mathrm{d}^{2}w}{\mathrm{d}z^{2}}+\frac{2\alpha}{z}\frac{\mathrm{d}w}{\mathrm{d}z}-\left(z+a+\frac{q}{z}\right)w=0, \\
&q=\lim_{x\rightarrow a}\frac{\mathrm{d}}{\mathrm{d}x}\log\tau(x)=b,
\end{aligned}\right.
\end{equation}
when $a$ is a zero of $y(x)$ with coefficient $\varepsilon=\varepsilon_+= 2\alpha$ in \eqref{P34ExpandAtzero}.
When $a$ is a zero with coefficient $\varepsilon=\varepsilon_{-}= -2\alpha$, taking $x\rightarrow a$  and the transformation  $\Phi_{1}=z^{\alpha+1}w$  in \eqref{Phi1P34} we arrive at 
\begin{equation}\label{RBHE3}
\left\{\begin{aligned}
&\frac{\mathrm{d}^{2}w}{\mathrm{d}z^{2}}+\frac{2(\alpha+1)}{z}\frac{\mathrm{d}w}{\mathrm{d}z}-\left(z+a+\frac{q}{z}\right)w=0, \\
&q=\lim_{x\rightarrow a}\frac{\mathrm{d}}{\mathrm{d}x}\log\tau(x)=b.
\end{aligned}\right.
\end{equation}

\subsection{Isomonodromy set of  accessory parameters of RBHE}

Thus,
we  see that  RBHE can be obtained  as a limit of the first row of  the isomonodromy family of  $\Phi(z,x)$   when  $x\to a$, $a$ being a zero or pole of  the  Painlev\'e XXXIV transcendents corresponding to the same monodromy data as RBHE.  While  the accessory parameters $(a,q)$ are expressed in terms of the parameters in the  Taylor or Laurent expansion  of the  PXXXIV   transcendents. As a consequence, we have the following description of the isomonodromy    accessory parameters  set of RBHE.


\begin{thm}\label{thm:IsoSetRBHE}
There is a discrete set of pairs of accessory parameters $(a_n,q_n)$  such that the RBHE  \eqref{int:RBHE} corresponding to these parameters has the same monodromy data \eqref{eq:RBHEMonDat} as the original
equation with the parameters $a$ and $q$.  This set coincides with the set of parameters $(a_n, b_n)$ in the Taylor expansion  \eqref{P34ExpandAtzero} near the zeros $a_n$ of the unique solution  of the PXXXIV equation \eqref{P34} with the same monodromy data \eqref{eq:RBHEMonDat} as the  RBHE.
\end{thm}

\section{ Accessory parameters of DHE }\label{sec:DHE}

\subsection{Monodromy of DHE}

Consider the DHE \eqref{int:DHE} with  the accessory parameters $a$ and $q$ and the fixed parameters
$\gamma$ and $p$.
Let $\Omega_k^{(\infty)}$, $k=1,2$, be the Stokes sectors
\begin{equation}\label{eq:DHEStokesSecInf} \Omega_k^{(\infty)}=\left\{z\in \mathbb{C}:~~\frac{\pi}{2}(-2k+1)<\arg z<\frac{\pi}{2}(-2k+5)\right\}.
\end{equation}
There exist unique linear independent solutions of \eqref{int:DHE}, namely $Y_k^{(\infty)}=(y_{k1}^{(\infty)}, y_{k2}^{(\infty)})$, which satisfy the normalized asymptotic behavior as $z\to \infty$
\begin{equation}\label{eq:DHE-infty}Y_k^{(\infty)}(z)\sim e^{-\frac{z}{4}}z^{-\frac{1}{2}\theta_0}\left(1,~ 1 /z\right) e^{\frac{z}{4}\sigma_3}z^{-\frac{1}{2}\theta_{\infty}\sigma_3}, \quad z\in \Omega_k^{(\infty)}, \end{equation}
for $k=1,2$. Here  the parameters $\theta_0=\gamma-1$ and $\theta_{\infty}=4p-\gamma+1$. 
The solutions  are related to each other  through the Stokes matrices
\begin{equation}\label{eq:DHE-StokesInf1}Y_2^{(\infty)}(z)
=Y_1^{(\infty)}(z)S_1^{(\infty)},\quad  S_1^{(\infty)}=\begin{pmatrix}
    1 & 0 \\
    s_1^{(\infty)} & 1
\end{pmatrix}, \end{equation}
\begin{equation}\label{eq:DHE-StokesInf2}Y_1^{(\infty)}(ze^{2\pi i})e^{\pi i\theta_{0}}e^{\pi i\theta_{\infty}\sigma_3}=Y_2^{(\infty)}(z)S_2^{(\infty)},\quad S_2^{(\infty)}=\begin{pmatrix}
    1 & s_2^{(\infty)} \\
  0 & 1
\end{pmatrix}. \end{equation}
Parallelly,  there exist unique linear independent solutions of  \eqref{int:DHE} determined by the  asymptotic behavior as $z\to 0$
\begin{equation}\label{eq:DHE-O}Y_k^{(0)}(z)\sim e^{-\frac{a^2}{4z}}z^{-\frac{1}{2}\theta_0}(1,~z) e^{\frac{a^2}{4z}\sigma_3}z^{\frac{1}{2}\theta_{0}\sigma_3}, \quad z\in \Omega_k^{(0)}, \end{equation}
where  $\Omega_k^{(0)}$ for $k=1,2$  are the Stokes sectors
\begin{equation}\label{eq:DHEStokesSecO} \Omega_k^{(0)}=\left\{z\in \mathbb{C}: ~~\frac{\pi}{2}(-2k+1)<\arg \left(  a^2 / z\right)<\frac{\pi}{2}(-2k+5)\right\}.
\end{equation}The solutions  are related to each other  by the Stokes matrices
\begin{equation}\label{eq:DHE-StokesO1}Y_2^{(0)}(z)=Y_1^{(0)}(z)S_1^{(0)},\quad  S_1^{(0)}=\begin{pmatrix}
    1 & 0 \\
    s_1^{(0)} & 1
\end{pmatrix}, \end{equation}
\begin{equation}\label{eq:DHE-StokesO2}Y_1^{(0)}(ze^{-2\pi i})e^{-\pi i\theta_{0}}e^{\pi i\theta_{0}\sigma_3}=Y_2^{(0)}(z)S_2^{(0)},\quad S_2^{(0)}=\begin{pmatrix}
    1 & s_2^{(0)} \\
  0 & 1
\end{pmatrix}. \end{equation}
Moreover,  the solutions $Y_1^{(\infty)}(z)$ and $Y_2^{(0)}(z)$ are connected  by
\begin{equation}\label{eq:DHE-E}Y_1^{(\infty)}(z)=Y_2^{(0)}(z) E,\end{equation}
where $E$ is some invertible constant matrix.
We define the monodromy data of  \eqref{int:DHE}  by the set
\begin{equation}\label{eq:DHEMonDat}\left\{e^{\pi i\theta_{\infty}},e^{\pi i\theta_0}; S^{(\infty)}_1, S^{(\infty)}_2; S^{(0)}_1, S^{(0)}_2; E\right\} .\end{equation}
Combining \eqref{eq:DHE-StokesInf1}, \eqref{eq:DHE-StokesInf2}, and \eqref{eq:DHE-StokesO1}-\eqref{eq:DHE-E}, we have the cyclic condition
 \begin{equation}\label{eq:DHE-CyclicCon}S_1^{(\infty)}S_2^{(\infty)}e^{-\pi i\theta_{\infty}\sigma_3}= E^{-1}\left(S_1^{(0)}\right)^{-1}e^{\pi i\theta_{0}\sigma_3} \left(S_2^{(0)}\right )^{-1} E. \end{equation}
The condition implies that the Stokes multipliers satisfy the equation
 \begin{equation}\label{eq:DHE-StokesEq}
e^{\pi i\theta_{\infty}}s_1^{(\infty)}s_2^{(\infty)}+2\cos(\pi \theta_{\infty})= e^{\pi i\theta_{0}} s_1^{(0)}s_2^{(0)}+2\cos(\pi \theta_{0}).\end{equation}

\subsection{Isomonodromy deformation and PIII equation} \label{IsoDefPIII}

To study the isomonodromy deformation of the  RTHE equation, it is  convenient to consider a $2\times 2$ matrix  system, which has two irregular singular points, both of rank $1$, at infinity and at the origin. The isomonodromy deformation of a matrix  system with such a singular pattern    has been considered in \cite{FIKN,Jim2}.

Recall the following Lax pair for PIII  (see \cite[ (C.18)-(C.20)]{Jim2}):
\begin{equation}\label{LaxPairP3}
\left\{\begin{aligned}
&\frac{\partial\Phi(z,x)}{\partial z}=A(z,x)\Phi(z,x),\\
&\frac{\partial\Phi(z,x)}{\partial x}=B(z,x)\Phi(z,x),
\end{aligned}\right.
\end{equation}
where
\begin{align*}
&A(z,x)=
\begin{pmatrix}
\frac{x}{4}-\frac{\theta_{\infty}}{2z}+\frac{4v-x}{4z^{2}} & \frac{r}{z}-\frac{uv}{z^{2}}\\
\frac{s}{z}+\frac{2v-x}{2uz^{2}} & -\frac{x}{4}+\frac{\theta_{\infty}}{2z}-\frac{4v-x}{4z^{2}}
\end{pmatrix},\\
&B(z,x)=
\begin{pmatrix}
\frac{z}{4}-\frac{4v-x}{4xz} & \frac{r}{x}+\frac{uv}{xz}\\
\frac{s}{x}-\frac{2v-x}{2xuz} & -\frac{z}{4}+\frac{4v-x}{4xz}
\end{pmatrix}.
\end{align*}
If we set $y=-r/uv$ and
$$\frac{\theta_{0}}{2}=\frac{v}{x}\left[\frac{\theta_{\infty}}{2}\left(4-\frac{x}{v}\right)+\frac{x-2v}{uv}r+2us\right],$$
then the compatibility condition of the above Lax pair reads
\begin{equation}\label{CCP3}
\left\{
\begin{aligned}
&x\frac{\mathrm{d}y}{\mathrm{d}x}=(4v-x)y^{2}+(2\theta_{\infty}-1)y+x,\\
&x\frac{\mathrm{d}v}{\mathrm{d}x}=-4yv^{2}+(2xy-2\theta_{\infty}+1)v+\frac{1}{2}(\theta_{0}+\theta_{\infty})x,\\
&x\frac{\mathrm{d}u}{\mathrm{d}x}=u\left(-\frac{x}{2v}(\theta_{0}+\theta_{\infty})-xy+\theta_{\infty}\right),
\end{aligned}
\right.
\end{equation}
which implies the PIII equation
\begin{align*}
\frac{\mathrm{d}^{2}y}{\mathrm{d}x^{2}}=\frac{1}{y}\left(\frac{\mathrm{d}y}{\mathrm{d}x}\right)^{2}-\frac{1}{x}\frac{\mathrm{d}y}{\mathrm{d}x}
+\frac{2\theta_{0}y^{2}+2-2\theta_{\infty}}{x}+y^{3}-\frac{1}{y}.
\end{align*}
The PIII $\tau$-function is defined by (see \cite[ (C.27)]{Jim2})
\begin{equation}\label{TauP3}
x\frac{\mathrm{d}}{\mathrm{d}x}\log\tau(x)=2y^{2}v^{2}+\left(-xy^{2}+2\theta_{\infty}y+x\right)v-\frac{(\theta_{0}+\theta_{\infty})xy}{2}
-\frac{x^{2}}{4}-\frac{\theta_{0}^{2}-\theta_{\infty}^{2}}{4}.\end{equation}
Using \eqref{CCP3} and \eqref{TauP3} to eliminate $v(x)$, we have
\begin{equation}\label{tauP3y}
x\frac{\mathrm{d}}{\mathrm{d}x}\log\tau(x)
=\frac{ x^2 y'^{2}}{8y^{2}}+\frac{xy'}{4y}-\frac{x^{2}}{8y^{2}}
-\frac{x^{2}y^{2}}{8}-\frac{\theta_{\infty}x}{2y}-\frac {\theta_0 x y} 2+\frac{1}{8}-\frac {\theta_0^2+\theta_\infty^2} 4.
\end{equation}

There exist unique solutions $\Phi_k^{(\infty)}(z,x)$ of the equation \eqref{LaxPairP3}, which satisfy the  normalized asymptotic behavior as $z\to \infty$
\begin{equation}\label{eq:DCHE-infty}\Phi_k^{(\infty)}(z,x)=(I+O(1/z)) e^{\frac{xz}{4}\sigma_3}z^{-\frac{1}{2}\theta_{\infty}\sigma_3}, \quad xz\in \Omega_k^{(\infty)}, \end{equation}
where $\Omega_k^{(\infty)}$ for $k=1,2$ are the Stokes sectors defined in \eqref{eq:DHEStokesSecInf}. The solutions are  related to each other by the Stokes matrices
\begin{equation}\label{eq:PIII-StokesInf1}\Phi_2^{(\infty)}(z)=\Phi_1^{(\infty)}(z)\hat{S}_1^{(\infty)}, \quad \hat{S}_1^{(\infty)}=\begin{pmatrix}
    1 & 0 \\
    \hat{s}_1^{(\infty)} & 1
\end{pmatrix}, \end{equation}
\begin{equation}\label{eq:PIII-StokesInf2}\Phi_1^{(\infty)}(ze^{2\pi i})e^{\pi i\theta_{\infty}\sigma_3}=\Phi_2^{(\infty)}(z)\hat{S}_2^{(\infty)}, \quad \hat{S}_2^{(\infty)}=\begin{pmatrix}
    1 & \hat{s}_2^{(\infty)} \\
  0 & 1
\end{pmatrix}. \end{equation}
Similarly, there  exist  unique solutions $\Phi_k^{(0)}(z,x)$ of the equation \eqref{LaxPairP3} such that
\begin{equation}\label{eq:DCHE-0}\Phi_k^{(0)}(z,x)=G_k(x)(I+O(z)) e^{\frac{x}{4z}\sigma_3}z^{\frac{1}{2}\theta_{0}\sigma_3}, \quad z\in \Omega_k^{(0)},~~z\to 0, \end{equation}
where $G_k(x)$ are matrices independent of $z$, and the sectors $\Omega_k^{(0)}$, $k=1,2$, are defined in \eqref{eq:DHEStokesSecO} with $\arg (a^2/z)$ replaced by  $\arg (x /z)$.
These solutions are  related to each other by the Stokes matrices
\begin{equation}\label{eq:PIII-StokesO1}\Phi_2^{(0)}(z)=\Phi_1^{(0)}(z)\hat{S}_1^{(0)}, \quad \hat{S}_1^{(0)}=\begin{pmatrix}
    1 & 0 \\
    \hat{s}_1^{(0)}& 1
\end{pmatrix}, \end{equation}
\begin{equation}\label{eq:PIII-StokesO2}\Phi_1^{(0)}(ze^{-2\pi i})e^{\pi i\theta_{0}\sigma_3}=\Phi_2^{(0)}(z)\hat{S}_2^{(0)}, \quad \hat{S}_2^{(0)}=\begin{pmatrix}
    1 & \hat{s}_2^{(0)} \\
  0 & 1
\end{pmatrix}. \end{equation}
Moreover,  the solutions $\Phi_1^{(\infty)}(z)$ and $\Phi_2^{(0)}(z)$ are connected  through
\begin{equation}\label{eq:PIII-E}\Phi_1^{(\infty)}(z)=\Phi_2^{(0)}(z) \hat{E},\end{equation}
where $\hat{E}$ is some invertible constant matrix.
The monodromy data of the system  \eqref{LaxPairP3} is defined as the set
\begin{equation}\label{eq:P3MonDat}\left\{e^{\pi i\theta_{\infty}},e^{\pi i\theta_0}; \hat{S}^{(\infty)}_1,  \hat{S}^{(\infty)}_2; \hat{S}^{(0)}_1, \hat{S}^{(0)}_2;\hat{E}\right \} .\end{equation}
Applying \eqref{eq:PIII-StokesInf1}, \eqref{eq:PIII-StokesInf2}, and \eqref{eq:PIII-StokesO1}-\eqref{eq:PIII-E}, we find the cyclic condition for the monodromy  data
 \begin{equation}\label{eq:PIII-CyclicCon}\hat{S}_1^{(\infty)}\hat{S}_2^{(\infty)}e^{-\pi i\theta_{\infty}\sigma_3}= \hat{E}^{-1}\left(\hat{S}_1^{(0)}\right)^{-1}e^{\pi i\theta_{0}\sigma_3} \left(\hat{S}_2^{(0)}\right )^{-1}\hat{E}. \end{equation}
Thus, the Stokes multipliers satisfy the equation
 \begin{equation}\label{eq:PIII-StokesEq}
e^{\pi i\theta_{\infty}}\hat{s}_1^{(\infty)}\hat{s}_2^{(\infty)}+2\cos(\pi \theta_{\infty})= e^{\pi i\theta_{0}}\hat{s}_1^{(0)}\hat{s}_2^{(0)}+2\cos(\pi \theta_{0}).\end{equation}

\subsection{Reduction of the  linear system   for PIII to  DHE }\label{sec:DHEP3}

In this section, we will show that the DHE  \eqref{int:DHE} can be obtained  as a limit of the  linear system \eqref{Phi1}  associated with   the Lax pair  for the Painlev\'{e} III equation as $x\to a$, where $a$ is  one of the poles or zeros of the solution to the
Painlev\'e  III equation.

Substituting the elements of $A(z,x)$ into \eqref{Phi1} gives
\begin{equation}\label{Phi1P3}
\frac{\mathrm{d}^{2}\Phi_{1}}{\mathrm{d}z^{2}}-\left(\frac{1}{z+1/y}-\frac{2}{z}\right)\frac{\mathrm{d}\Phi_{1}}{\mathrm{d}z}-Q(z,x)\Phi_{1}=0,
\end{equation}
where
\begin{align*}
Q(z,x)&=\frac{x^{2}}{16}+\frac{(2-\theta_{\infty})x}{4z}-\frac{\theta_{0}x}{4z^{3}}+\frac{x^{2}}{16z^{4}}-\frac{1}{z+1/y}\left(\frac{x}{4}-\frac{\theta_{\infty}}{2z}
+\frac{v-\frac{x}{4}}{z^{2}}\right)\\
&\ \ \ +\frac{1}{z^{2}}\left[y^{2}v^{2}+\frac{1}{2}(-xy^{2}+2\theta_{\infty}y+x)v-\frac{1}{4}(\theta_{0}+\theta_{\infty})xy
-\frac{x^{2}}{8}+\frac{\theta_{\infty}^{2}}{4}-\frac{\theta_{\infty}}{2}\right] \\ &=\frac{x^{2}}{16}+\frac{(2-\theta_{\infty})x}{4z}-\frac{\theta_{0}x}{4z^{3}}+\frac{x^{2}}{16z^{4}}-\frac{1}{z+1/y}\left(\frac{x}{4}-\frac{\theta_{\infty}}{2z}
+\frac{v-\frac{x}{4}}{z^{2}}\right)\\
&\ \ \ +\frac{1}{z^{2}}\left[\frac{1}{2}x\frac{\mathrm{d}}{\mathrm{d}x}\log\tau(x)+\frac{\theta_{0}^{2}+\theta_{\infty}^{2}}{8}-\frac{\theta_{\infty}}{2}\right].
\end{align*}

To obtain DHE, we eliminate the extra singularity  $z=-1/y$ in the above equation by considering the poles and zeros of $y$.
It is known that $y$ admits the following Laurent expansion near a pole
\begin{equation}\label{PIIIExpandAtpole}
y(x)=\frac{\varepsilon}{x-a}-\frac{2\theta_{0}+\varepsilon}{2a}+b(x-a)+O((x-a)^{2}),\end{equation}
and the Taylor expansion near a zero
\begin{equation}\label{PIIIExpandAtzero}
y(x)=\sigma(x-a)+\frac{\sigma-2+2\theta_{\infty}}{2a}(x-a)^{2}+b(x-a)^{3}+O((x-a)^{4}),
\end{equation}
where $\varepsilon=\varepsilon_{\pm}=\pm1$, $\sigma=\sigma_{\pm}=\pm1$, $a\neq0 $ and $b$ is arbitrary.

In view of the first equation of compatibility condition \eqref{CCP3}, \eqref{PIIIExpandAtpole} and \eqref{PIIIExpandAtzero}, we obtain that
\begin{equation}\label{vPIIIExpand}
v(x)=\left\{\begin{aligned}
  &O(x-a), & \varepsilon=\varepsilon_{+},\\
  &\textstyle\frac{a}{2}+O(x-a),   & \varepsilon=\varepsilon_{-},\\
  &O(1), & \sigma=\sigma_{+},\\
  &\textstyle-\frac{a}{2}(x-a)^{-2}+O(1),   & \sigma=\sigma_{-}.
\end{aligned} \right.
\end{equation}

Thus, taking $x\rightarrow a$ in \eqref{Phi1P3} and setting
$$
\Phi_{1}=z^{\frac{\theta_{0}}{2}}e^{\frac{a}{4}(z+\frac{1}{z})}w,\ \ \ s=az,
$$
we obtain the DHEs
\begin{equation}\label{DHE1}
\left\{
\begin{aligned}
&\frac{\mathrm{d}^{2}w}{\mathrm{d}s^{2}}+\left(\frac{1}{2}+\frac{1+\theta_{0}}{s}-\frac{a^{2}}{2s^{2}}\right)\frac{\mathrm{d}w}{\mathrm{d}s}+\frac{ps-q}{s^{2}}w=0,\\
&p=\frac{1}{4}(\theta_{\infty}+\theta_{0}),\\
&q=\lim_{x\rightarrow a}\frac{1}{2}x\frac{\mathrm{d}}{\mathrm{d}x}\log\tau(x)+\frac{a^{2}}{8}-\frac{\theta_{0}^{2}-\theta_{\infty}^{2}}{8}\\
&~~=-\frac{3}{8}a^2 b+\frac{a^2}{8}+\frac{1}{32}(2\theta_{0}+1)(2\theta_{0}+5) 
\end{aligned}\right.
\end{equation}
for $\varepsilon=\varepsilon_{+}$,
\begin{equation}\label{DHE2}
\left\{
\begin{aligned}
&\frac{\mathrm{d}^{2}w}{\mathrm{d}s^{2}}+\left(\frac{1}{2}+\frac{2+\theta_{0}}{s}-\frac{a^{2}}{2s^{2}}\right)\frac{\mathrm{d}w}{\mathrm{d}s}+\frac{ps-q}{s^{2}}w=0,\\
&p=\frac{1}{4}(\theta_{\infty}+\theta_{0}),\\
&q=\lim_{x\rightarrow a}\frac{1}{2}x\frac{\mathrm{d}}{\mathrm{d}x}
\log\tau(x)+\frac{a^{2}}{8}-\frac{\theta_{0}^{2}-
\theta_{\infty}^{2}}{8}-\frac{\theta_{0}+\theta_{\infty}}{2} \end{aligned}
\right.
\end{equation}
for $\sigma=\sigma_{+}$, and
\begin{equation}\label{DHE3}
\left\{
\begin{aligned}
&\frac{\mathrm{d}^{2}w}{\mathrm{d}s^{2}}+\left(\frac{1}{2}+\frac{2+\theta_{0}}{s}-\frac{a^{2}}{2s^{2}}\right)\frac{\mathrm{d}w}{\mathrm{d}s}+\frac{ps-q}{s^{2}}w=0,\\
&p=\frac{1}{4}(\theta_{\infty}+\theta_{0}+2),\\
&q=\lim_{x\rightarrow a}\frac{1}{2}\left(x\frac{\mathrm{d}}{\mathrm{d}x}\log\tau(x)-\frac{a}{x-a}\right)+\frac{a^{2}}{8}-\frac{\theta_{0}^{2}-\theta_{\infty}^{2}}{8}-\frac{\theta_{0}}{2}-\frac{3}{4}
\end{aligned}
\right.
\end{equation}
for $\sigma=\sigma_{-}$. For $\varepsilon=\varepsilon_{-}$ in \eqref{Phi1P3}, taking $x\rightarrow a$ and setting
$$
\Phi_{1}=z^{-\frac{\theta_{0}}{2}}e^{\frac{a}{4}(z-\frac{1}{z})}w,\ \ \ s=az,
$$
we obtain the DHE
\begin{equation}\label{DHE4}
\left\{
\begin{aligned}
&\frac{\mathrm{d}^{2}w}{\mathrm{d}s^{2}}+\left(\frac{1}{2}+\frac{1-\theta_{0}}{s}+\frac{a^{2}}{2s^{2}}\right)\frac{\mathrm{d}w}{\mathrm{d}s}+\frac{ps-q}{s^{2}}w=0,\\
&p=\frac{1}{4}(\theta_{\infty}-\theta_{0}),\\
&q=\lim_{x\rightarrow a}\frac{1}{2}x\frac{\mathrm{d}}{\mathrm{d}x}\log\tau(x)-\frac{a^{2}}{8}-\frac{\theta_{0}^{2}-\theta_{\infty}^{2}}{8}.
\end{aligned}
\right.
\end{equation}

%
%

\subsection{Isomonodromy set of  accessory parameters of DHE}

Consider the DHE  \eqref{int:DHE} with  the accessory parameters $a$ and $q$ and the fixed parameters
$\gamma$ and $p$. 
We  have shown that  the DHE can be obtained  as a limit of the first row of  the isomonodromy family of  $\Phi(z,x)$,  when $x$ tends to  one of the poles or zeros of  the  Painlev\'e III transcendents corresponding to the same monodromy data of the DHE.  While  the accessory parameters $(a,q)$ are expressed in terms of the parameters in the Laurent or Taylor expansion  of the  PIII transcendents with the parameters $\theta_0=\gamma-1$ and $\theta_{\infty}=4p-\gamma+1$. As a result, we have the following description of the isomonodromy set of  accessory parameters of DHE. 

\begin{thm}\label{thm:IsoSetDHE}
There is a discrete set of pairs of accessory parameters $(a_n,q_n)$  such that the DHE  \eqref{int:DHE}
corresponding to these parameters has the same monodromy data as the original
equation with the parameters $a$ and $q$. Under the bijection given in the last equation of \eqref{DHE1},  this set coincides with the set of parameters $(a_n, b_n)$ in the  Laurent expansion near the poles $a_n$ of the unique solution  of PIII with the same monodromy data \eqref{eq:DHEMonDat} as the  DHE.
\end{thm}

\section{Accessory  parameters of BHE}\label{sec:BHE}

\subsection{Monodromy of BHE}
Consider the BHE  \eqref{int:BHE} with  the accessory parameters $a$ and $q$, the other fixed parameters 
$\gamma$ and $p$.
There exist unique  solutions of  \eqref{int:BHE} which satisfy the normalized asymptotic behavior as $z\to \infty$
\begin{equation}\label{eq:BHE-infty}Y_k(z)\sim e^{-\frac{z^2}{2}-az}z^{-(\theta_0+1)}(1,1) e^{\frac{z^2+2az}{2}\sigma_3}z^{-\theta_{\infty}\sigma_3}, \quad z\in \Omega_k, \end{equation}
where the parameters $\theta_0=(\gamma-1)/2$, $\theta_{\infty}=(p-\gamma-1)/2$ and 
the Stokes sectors $\Omega_k$, $k=1,2,3,4$  are given by
\begin{equation}\label{eq:BHEStokesSec} \Omega_k=\left\{z\in \mathbb{C}: ~~\frac{\pi}{4}(2k-5)<\arg z<\frac{\pi}{4}(2k-1)\right\}.
\end{equation}

The solutions are  related to each other by the Stokes matrices
\begin{equation}\label{eq:BHE-Stokes1}Y_{k+1}(z)=Y_k(z)S_k, \quad  k=1,2,3, \end{equation}
\begin{equation}\label{eq:P4-Stokes2}Y_1(z)=Y_4(ze^{2\pi i})S_4e^{2\pi i\theta_0}e^{2\pi i\theta_{\infty}\sigma_3}, \end{equation}
where
\begin{equation}\label{eq:BHE-StokesM} S_{2k-1}=\begin{pmatrix}
    1 & 0 \\
    s_{2k-1} & 1
\end{pmatrix}, \quad S_{2k}=\begin{pmatrix}
    1 & s_{2k} \\
  0 & 1
\end{pmatrix}, \quad k=1,2. \end{equation}
The solution $Y_1(z)$  has the asymptotic behavior  as $z\to 0$
\begin{equation}\label{eq:BHEZero}Y_1(z) \sim z^{-\theta_0}(1,1)z^{\theta_0\sigma_3}E ,\end{equation}
with  some invertible constant matrix $E$.
The monodromy data of BHE  \eqref{int:BHE} is defined as the set
\begin{equation}\label{eq:BHEMonDat}\left\{e^{2\pi i\theta_{\infty}},e^{2\pi i\theta_0}; S_1,  S_2, S_3, S_4; E\right \} .\end{equation}
By \eqref{eq:BHE-Stokes1}-\eqref{eq:BHEZero}, they satisfy the cyclic  condition
 \begin{equation}\label{eq:BHE-CyclicCon}S_1S_2S_3S_4e^{2\pi i\theta_{\infty}\sigma_3}= E^{-1}e^{-2\pi i\theta_0\sigma_3} E. \end{equation}
 Thus, the Stokes multipliers  $s_k$, $k=1,2,3,4$, satisfy the algebraic equation
 \begin{equation}\label{eq:BHE-StokesEq}
e^{2\pi i\theta_{\infty}}(1+s_2s_3)+e^{-2\pi i\theta_{\infty}}\left [s_1s_4+(1+s_1s_2)(1+s_3s_4)\right ]=2\cos(2\pi \theta_{0}).\end{equation}
It is direct to see that the cyclic condition  \eqref{eq:BHE-CyclicCon} also specifies the connection matrix $E$ to within a left-multiplicative diagonal matrix.

\subsection{Isomonodromy deformation and PIV equation} \label{IsoDefPII}

To study the isomonodromy deformation of the  BHE, it is  convenient to consider a $2\times 2$ matrix  system, which has one irregular singular point of rank 2 at infinity and one regular singularity at the origin. In \cite{FIKN,Jim2}, the isomonodromy deformation of such a matrix  system has been considered.

Consider the following Lax pair for PIV (see \cite[ (C.30)-(C.31)]{Jim2}):
\begin{equation}\label{LaxPairP4}
\left\{\begin{aligned}
&\frac{\partial\Phi(z,x)}{\partial z}=A(z,x)\Phi(z,x),\\
&\frac{\partial\Phi(z,x)}{\partial x}=B(z,x)\Phi(z,x),
\end{aligned}\right.
\end{equation}
where
\begin{align*}
A(z,x)&=
\begin{pmatrix}
z+x+\frac{\theta_{0}-v}{z} & u\left(1-\frac{y}{2z}\right)\\
\frac{2}{u}\left(v-\theta_{0}-\theta_{\infty}+\frac{v^{2}-2\theta_{0}v}{yz}\right) & -z-x-\frac{\theta_{0}-v}{z}
\end{pmatrix},\\
B(z,x)&=
\begin{pmatrix}
z & u\\
\frac{2}{u}(v-\theta_{0}-\theta_{\infty}) & -z
\end{pmatrix}.
\end{align*}
The compatibility condition of the Lax pair gives
\begin{equation}\label{CCP4}
\left\{
\begin{aligned}
\frac{\mathrm{d}u}{\mathrm{d}x}&=-u(y+2x),\\
\frac{\mathrm{d}y}{\mathrm{d}x}&=-4v+y^{2}+2xy+4\theta_{0},\\
\frac{\mathrm{d}v}{\mathrm{d}x}&=-(v-\theta_{0}-\theta_{\infty})y-\frac{2v(v-2\theta_{0})}{y},
\end{aligned}
\right.
\end{equation}
which leads to the PIV equation
\begin{equation}\label{eq:P4}
\frac{\mathrm{d}^{2}y}{\mathrm{d}x^{2}}=\frac{1}{2y}\left(\frac{\mathrm{d}y}{\mathrm{d}x}\right)^{2}+\frac{3}{2}y^{3}+4xy^{2}+2(x^{2}+1-2\theta_{\infty})y-\frac{8\theta_{0}^{2}}{y}.
\end{equation}
The PIV $\tau$-function is defined by (see \cite[(C.35)]{Jim2})
\begin{equation}\label{TauP4}
\frac{\mathrm{d}}{\mathrm{d}x}\log \tau(x)=\frac{2}{y}v^{2}-\bigg(y+2x+\frac{4\theta_{0}}{y}\bigg)v+(\theta_{0}+\theta_{\infty})(y+2x).
\end{equation}

There exist unique solutions $\Phi_k(z,x)$ of the equation \eqref{LaxPairP4} subject to the  normalized asymptotic behavior as $z\to \infty$
\begin{equation}\label{eq:P4-infty}\Phi_k(z,x)=(I+O(1/z)) e^{\frac{z^2+2xz}{2}\sigma_3}z^{-\theta_{\infty}\sigma_3}, \quad z\in \Omega_k, \end{equation}
where   $\Omega_k$, $k=1,2,3,4$,  are the Stokes sectors defined in
\eqref{eq:BHEStokesSec}. Using \eqref{eq:P4-infty}, it is seen that the solutions are  related to each other by the Stokes matrices
\begin{equation}\label{eq:P4-Stokes1}\Phi_{k+1}(z)=\Phi_k(z)\hat{S}_k, \quad  k=1,2,3, \end{equation}
\begin{equation}\label{eq:P4-Stokes2}\Phi_1(z)=\Phi_4(ze^{2\pi i})\hat{S}_4e^{2\pi i\theta_{\infty}\sigma_3}, \end{equation}
where
\begin{equation}\label{eq:P4-StokesM} \hat{S}_{2k-1}=\begin{pmatrix}
    1 & 0 \\
    \hat{s}_{2k-1} & 1
\end{pmatrix}, \quad \hat{S}_{2k}=\begin{pmatrix}
    1 & \hat{s}_{2k} \\
  0 & 1
\end{pmatrix},~~ k=1,2. \end{equation}
The solution $\Phi_1(z,x)$  has the asymptotic behaviors as $z\to 0$
$$\Phi_1(z,x) =\Phi_1^{(0)}(x)(I+O(z))z^{\theta_0\sigma_3}\hat{E},$$
with  some invertible constant matrix $\hat{E}$.
The monodromy data of the system  \eqref{LaxPairP4} is defined as the set
\begin{equation}\label{eq:P4MonDat}\left\{e^{2\pi i\theta_{\infty}},e^{\pi i\theta_0}; \hat{S}_1, \hat{S}_2, \hat{S}_3, \hat{S}_4; \hat{E}\right\} .\end{equation}
They satisfy the cyclic  condition
 \begin{equation}\label{eq:P4-CyclicCon}\hat{S}_1\hat{S}_2\hat{S}_3\hat{S}_4e^{2\pi i\theta_{\infty}\sigma_3}= \hat{E}^{-1}e^{-2\pi i\theta_0\sigma_3} \hat{E}. \end{equation}
 Thus, the Stokes multipliers  $\hat{s}_k$, $k=1,2,3,4$ satisfy
 \begin{equation}\label{eq:P4-StokesEq}
e^{2\pi i\theta_{\infty}}(1+\hat{s}_2\hat{s}_3)+e^{-2\pi i\theta_{\infty}}(\hat{s}_1\hat{s}_4+(1+\hat{s}_1\hat{s}_2)(1+\hat{s}_3\hat{s}_4))=2\cos(2\pi \theta_{0}).\end{equation}
Moreover, the cyclic condition \eqref{eq:P4MonDat} specifies the connection matrices $\hat{E}$ to within a left-multiplicative diagonal matrices; see \cite{FIKN}.

\subsection{Reduction of the  linear system  for PIV to  BHE}\label{sec:BHEP4}

In this subsection, we will derive BHE from  \eqref{Phi1} associated with the isomonodromy system of PIV when the PIV transcendents $y(x)$ tends to zero or infinity.

Substituting  $A(z,x)$ into \eqref{Phi1} gives
\begin{equation}\label{Phi1P4}
\frac{\mathrm{d}^{2}\Phi_{1}}{\mathrm{d}z^{2}}-\left(\frac{1}{z-\lambda}-\frac{1}{z}\right)\frac{\mathrm{d}\Phi_{1}}{\mathrm{d}z}-Q(z,x)\Phi_{1}=0,
\end{equation}
where
\begin{align*}
Q(z,x)&=z^{2}+2xz+x^{2}+2(1-\theta_{\infty})+\frac{\theta_{0}^{2}}{z^{2}}-\frac{1}{z-\lambda}\left(z+x+\frac{\theta_{0}-v}{z}\right)\\
&\ \ \ \ \ \ \ +\frac{1}{z}\left[\frac{2}{y}v^{2}-\bigg(y+2x+\frac{4\theta_{0}}{y}\bigg)v+(\theta_{0}+\theta_{\infty})y+(1+2\theta_{0})x\right]\\
&=z^{2}+2xz+x^{2}+2(1-\theta_{\infty})+\frac{\theta_{0}^{2}}{z^{2}}-\frac{1}{z-\lambda}\left(z+x+\frac{\theta_{0}-v}{z}\right)\\
&\ \ \ \ \ \ \ +\frac{1}{z}\left[\frac{\mathrm{d}}{\mathrm{d}x}\log\tau(x)+(1-2\theta_{\infty})x\right].
\end{align*}

To derive BHE, we need to eliminate the extra singularity $\lambda=y/2$. Thus we  consider the zeros and poles of $y$.
According to  \cite[(1.4)-(1.5)]{Mas}, the solution $y$ of PIV equation admits the  Taylor expansion near a zero
\begin{equation}\label{P4Expandzero}
y(x)=\varepsilon(x-a)+b(x-a)^{2}+O((x-a)^{3}),\end{equation}
 the Laurent expansion near a pole
\begin{equation}\label{P4Expandpole}
y(x)=\sigma(x-a)^{-1}-a+a_{1}(x-a)+b(x-a)^{2}+O((x-a)^{3}),
\end{equation}
where $b$ is arbitrary and the parameters
$$
\varepsilon=\varepsilon_{\pm}=\pm4\theta_{0},\ \ \
\sigma=\sigma_{\pm}=\pm1,\ \ \
a_{1}=\frac{1}{3}\sigma\left(a^{2}-2+4(\theta_{\infty}-\sigma)\right).
$$

Let us first consider the zeros of $y$. It follows from \eqref{CCP4} and \eqref{P4Expandzero}
 that
$$
v(x)=\left\{\begin{aligned}
  &O(x-a), &  \varepsilon=\varepsilon_{+},\\
  &2\theta_{0}+O(x-a),   &  \varepsilon=\varepsilon_{-}.
\end{aligned} \right.
$$
For $\varepsilon=\varepsilon_{+}$, taking $x\rightarrow a$ in \eqref{Phi1P4} and then setting
$$
\Phi_{1}=z^{\theta_{0}}e^{\frac{1}{2}z^{2}+az}w,
$$
we obtain the BHE
\begin{equation}\label{BHE1}
\left\{
\begin{aligned}
&\frac{\mathrm{d}^{2}w}{\mathrm{d}z^{2}}+\left(2z+2a+\frac{2\theta_{0}}{z}\right)\frac{\mathrm{d}w}{\mathrm{d}z}+\frac{pz-q}{z}w=0,\\
&p=2(\theta_{0}+\theta_{\infty}),\\
&q=\lim_{x\rightarrow a}\frac{\mathrm{d}}{\mathrm{d}x}\log\tau(x)-2(\theta_{0}+\theta_{\infty})a.
\end{aligned}
\right.
\end{equation}
For $\varepsilon=\varepsilon_{-}$, by taking $x\rightarrow a$  in \eqref{Phi1P4} and setting
$$
\Phi_{1}=z^{\theta_{0}+1}e^{\frac{1}{2}z^{2}+az}w,
$$
we obtain the BHE
\begin{equation}\label{BHE2}
\left\{
\begin{aligned}
&\frac{\mathrm{d}^{2}w}{\mathrm{d}z^{2}}+\left(2z+2a+\frac{2\theta_{0}+2}{z}\right)\frac{\mathrm{d}w}{\mathrm{d}z}+\frac{pz-q}{z}w=0,\\
&p=2(\theta_{0}+\theta_{\infty}+1),\\
&q=\lim_{x\rightarrow a}\frac{\mathrm{d}}{\mathrm{d}x}\log\tau(x)-2(\theta_{0}+\theta_{\infty}+1)a.
\end{aligned}
\right.
\end{equation}

Next we focus on the poles of $y$.  It is seen from \eqref{CCP4} and \eqref{P4Expandpole} that
\begin{equation}\label{vP4Expand}
v(x)=\left\{\begin{aligned}
  &\textstyle\frac{1}{2}(x-a)^{-2}+O(1), &  \sigma=\sigma_{+},\\
  &\theta_{0}+\theta_{\infty}+O(x-a),   &  \sigma=\sigma_{-}.
\end{aligned} \right.
\end{equation}
By taking the limit $x\rightarrow a$ in \eqref{Phi1P4} and then making the transformation
$$
\Phi_{1}=z^{\theta_{0}}e^{\frac{1}{2}z^{2}+az}w,
$$
we obtain another two BHEs
\begin{equation}\label{BHE3}
\left\{
\begin{aligned}
&\frac{\mathrm{d}^{2}w}{\mathrm{d}z^{2}}+\left(2z+2a+\frac{2\theta_{0}+1}{z}\right)\frac{\mathrm{d}w}{\mathrm{d}z}+\frac{pz-q}{z}w=0,\\
&p=2(\theta_{0}+\theta_{\infty}+1),\\
&q=\lim_{x\rightarrow a}\left(\frac{\mathrm{d}}{\mathrm{d}x}\log\tau(x)-\frac{1}{x-a}\right)-(2\theta_{0}+2\theta_{\infty}+1)a\\
&~~=-b-\left(2\theta_{0}+2\theta_{\infty}-\frac{1}{2}\right)a
\end{aligned}
\right.
\end{equation}
for $\sigma=\sigma_{+}=1$,
\begin{equation}\label{BHE4}
\left\{
\begin{aligned}
&\frac{\mathrm{d}^{2}w}{\mathrm{d}z^{2}}+\left(2z+2a+\frac{2\theta_{0}+1}{z}\right)\frac{\mathrm{d}w}{\mathrm{d}z}+\frac{pz-q}{z}w=0,\\
&p=2(\theta_{0}+\theta_{\infty}),\\
&q=\lim_{x\rightarrow a}\frac{\mathrm{d}}{\mathrm{d}x}\log\tau(x)-2(\theta_{0}+\theta_{\infty})a\\
&~=-b-\left(2\theta_{0}+2\theta_{\infty}+\frac{1}{2}\right)a
\end{aligned}
\right.
\end{equation}
for $\sigma=\sigma_{-}=-1$.


\subsection{Isomonodromy set of  accessory parameters of BHE}

We have shown that  the BHE \eqref{int:BHE},  with  the accessory parameters $a$ and $q$, the other fixed parameters $\gamma$ and $p$
 can be obtained  as a limit of the first row of  the isomonodromy family of  $\Phi(z,x)$,  when  $x\to a$, $a$ being one of the poles or zeros of  the  Painlev\'e IV transcendents.  While  the accessory parameters $(a,q)$ are expressed in terms of the parameters in the  Laurent or Taylor expansion  of the  PIV transcendents as given in \eqref{BHE1}-\eqref{BHE4}. Therefore, we have the following description of the isomonodromy set of  accessory parameters of BHE.

\begin{thm}\label{thm:IsoSetBHE}
There is a discrete set of pairs of accessory parameters $(a_n,q_n)$  such that  the BHE \eqref{int:BHE}
corresponding to these parameters has the same monodromy data described by \eqref{eq:BHE-StokesEq} as the original
equation with the parameters $a$ and $q$. Under the  bijection given in the last equation of \eqref{BHE3},  this set coincides with the set of parameters $(a_n, b_n)$ in the  Laurent expansion near the poles $a_n$ of the unique solution  of PIV  \eqref{eq:P4} with the parameter $\theta_0=(\gamma-1)/2$, $\theta_{\infty}=(p-\gamma-1)/2$ and  the  same monodromy data  \eqref{eq:BHE-StokesEq} as the  BHE.
\end{thm}

\section{ Accessory parameters of CHE }\label{sec:CHE}

\subsection{Monodromy of CHE }

Consider the CHE  \eqref{int:CHE}  
with  the accessory parameters $a$ and $q$, the fixed parameters $\gamma$, $\delta$ and $p$.
There exist uniquely two linear independent solutions of  \eqref{int:CHE}, namely $Y_k=(y_{k1}, y_{k2})$,   normalizing the asymptotic behavior as $z\to \infty$
\begin{equation}\label{eq:CHE-infty}Y_k(z)\sim e^{-\frac{z}{2}}z^{-\frac{1}{2}(\theta_0+\theta_1)}(1,~1/z) e^{\frac{z}{2}\sigma_3}z^{-\frac{1}{2}\theta_{\infty}\sigma_3}, \quad z\in \Omega_k, \end{equation}
where the parameters $\theta_0=\gamma$, $\theta_1=\delta-1$, $\theta_{\infty}=2p-\gamma-\delta+1$ and the Stokes sectors $\Omega_k$, $k=1,2$ are defined by
\begin{equation}\label{eq:CHEStokesSec} \Omega_k=\left\{z\in \mathbb{C}:~~ \frac{\pi}{2}(-2k+1)<\arg z<\frac{\pi}{2}(-2k+5)\right\}.
\end{equation}
The solutions  are related to each other  by the Stokes matrices
\begin{equation}\label{eq:CHE-Stokes1}Y_2(z)=Y_1(z)S_1,\quad  S_1=\begin{pmatrix}
    1 & 0 \\
    s_1 & 1
\end{pmatrix}, \end{equation}
\begin{equation}\label{eq:CHE-Stokes2}Y_1(ze^{2\pi i})e^{\pi i(\theta_0+\theta_1) }e^{\pi i\theta_{\infty}\sigma_3}=Y_2(z)S_2,\quad S_2=\begin{pmatrix}
    1 & s_2 \\
  0 & 1
\end{pmatrix}. \end{equation}
Near the regular singular points, $Y_1$ has the asymptotic behaviors
$$Y_1(z) \sim \left(z^0, z^{1-\theta_0}\right)E_0, \quad Y_1(z) \sim \left((z-a)^{0}, (z-a)^{-\theta_1}\right )E_1,$$
with  some invertible constant matrices  $E_0$ and $E_1$. After an
 analytic continuation along a closed loop around a singular point,  we obtain the relations
\begin{equation}\label{eq:CHE-M}Y_1\left(z_k+e^{2\pi i}(z-z_k)\right )=Y_1(z) e^{-\pi i\theta_k}M_k,\quad M_{k}=E_k^{-1}e^{\pi i\theta_k\sigma_3} E_k\end{equation}
for $k=0,1$, where $z_0=0$ and $z_1=a$.
We define the monodromy data of the CHE equation as
\begin{equation}\label{eq:CHEMonDat}\left\{e^{\pi i\theta_{\infty}},e^{\pi i\theta_0}, e^{\pi i\theta_1}; S_1, S_1 ; E_0, E_1\right \} .\end{equation}
According to \eqref{eq:CHE-Stokes1}-\eqref{eq:CHE-M}, the monodromy matrices  satisfies the cyclic  condition
 \begin{equation}\label{eq:CHE-CyclicCon}M_1M_0=S_1S_2e^{-\pi i\theta_{\infty}\sigma_3}.\end{equation}
 Suppose
 \begin{equation}\label{eq:CHE-sigama}  2\cos (\pi \sigma)=\Tr M_1M_2=2\cos (\pi\theta_{\infty})+s_1s_2e^{\pi i\theta_{\infty}}, \end{equation}
 with $0\leq \Re \sigma\leq 1$.
 Then,  the connection matrices can be parameterized  in terms of $\sigma$ and the other free parameter $s$; see \cite{Jim1}.
 For instance, we have for $\sigma\neq 0$ \cite[Equation (3.7)]{Jim1}:
 \begin{align}\label{eq:CHE-E1}D_1E_1D&=\begin{pmatrix}
    \sin(\frac{\pi}{2}(\theta_1+\theta_0+\sigma))\sin(\frac{\pi}{2}(\theta_1-\theta_0+\sigma))& -s   \sin(\frac{\pi}{2}(\theta_1+\theta_0-\sigma))\sin(\frac{\pi}{2}(\theta_1-\theta_0-\sigma))\\
  -s^{-1}& 1\end{pmatrix}\nonumber\\
  &~~~~~~~~~~~~~~~~~~~~~~~~~~~~~~~~~~~~~~~~~~~~~~~~~~ \times   \begin{pmatrix}
      e^{-\frac{1}{2}\pi i\sigma} & \sin(\frac{\pi}{2}(\theta_{\infty}+\sigma)) \\
   e^{\frac{1}{2}\pi i \sigma} & \sin(\frac{\pi}{2}(\theta_{\infty}-\sigma))\end{pmatrix}
  , \end{align}
  and
 \begin{align}\label{eq:CHE-E0}D_0E_0D&=\begin{pmatrix}
    \sin(\frac{\pi}{2}(\theta_1+\theta_0+\sigma))& -s   e^{-\pi i\sigma} \sin(\frac{\pi}{2}(\theta_1+\theta_0-\sigma))\\
  -s^{-1}e^{\pi i\sigma} \sin(\frac{\pi}{2}(\theta_1-\theta_0+\sigma))& \sin(\frac{\pi}{2}(\theta_1-\theta_0-\sigma))\end{pmatrix}\nonumber\\
  &~~~~~~~~~~~~~~~~~~~~~~~~~~~~~~~~~~~~~~~~~~~~~~~~~~ \times   \begin{pmatrix}
      e^{-\frac{1}{2}\pi i\sigma} & \sin(\frac{\pi}{2}(\theta_{\infty}+\sigma)) \\
   e^{\frac{1}{2}\pi i \sigma} & \sin(\frac{\pi}{2}(\theta_{\infty}-\sigma))\end{pmatrix},\end{align}
   with some invertible diagonal matrices $D_1,D_0$ and $D$.

\subsection{Isomonodromy deformation and PV equation}

To study the isomonodromy deformation of the CHE, it is  convenient to consider a $2\times 2$ matrix  system, which has one irregular singular point of rank 1 at infinity and two regular singularities. In \cite{FIKN,Jim2}, the isomonodromy deformation of such a matrix  system has been considered.

Recall the following Lax pair for PV  (see \cite[(C.38)-(C.39)]{Jim2}):
\begin{equation}\label{LaxPairP5}
\left\{\begin{aligned}
&\frac{\partial\Phi(z,x)}{\partial z}=A(z,x)\Phi(z,x), \ \ A(z,x)=\frac{\sigma_{3}}{2}+\frac{A_{0}(x)}{z}+\frac{A_{1}(x)}{z-x}, \\
&\frac{\partial\Phi(z,x)}{\partial x}=B(z,x)\Phi(z,x), \ \ B(z,x)=-\frac{A_{1}(x)}{z-x},
\end{aligned}\right.
\end{equation}
the coefficients
\begin{align*}
A_{0}(x)&=
\begin{pmatrix}
     v+\frac{\theta_{0}}{2} & -u(v+\theta_{0}) \\
     \frac{v}{u} & -v-\frac{\theta_{0}}{2}
\end{pmatrix},\\
A_{1}(x)&=
\begin{pmatrix}
     -v-\frac{\theta_{0}+\theta_{\infty}}{2} & yu\big(v+\frac{\theta_{0}-\theta_{1}+\theta_{\infty}}{2}\big) \\
     -\frac{1}{yu}\big(v+\frac{\theta_{0}+\theta_{1}+\theta_{\infty}}{2}\big) & v+\frac{\theta_{0}+\theta_{\infty}}{2}
\end{pmatrix}.
\end{align*}
The compatibility condition of the above Lax pair reads
\begin{equation}\label{CCP5}
\left\{
\begin{aligned}
x\frac{\mathrm{d}y}{\mathrm{d}x}&=xy-2v(y-1)^{2}-(y-1)\left(\frac{\theta_{0}-\theta_{1}+\theta_{\infty}}{2}y
-\frac{3\theta_{0}+\theta_{1}+\theta_{\infty}}{2}\right),\\
x\frac{\mathrm{d}v}{\mathrm{d}x}&=yv\left(v+\frac{\theta_{0}-\theta_{1}+\theta_{\infty}}{2}\right)
-\frac{1}{y}(v+\theta_{0})\left(v+\frac{\theta_{0}+\theta_{1}+\theta_{\infty}}{2}\right),\\
x\frac{\mathrm{d}u}{\mathrm{d}x}&=u\left[-2v-\theta_{0}-\theta_{\infty}+y\left(v+\frac{\theta_{0}-\theta_{1}+\theta_{\infty}}{2}\right)
+\frac{1}{y}\left(v+\frac{\theta_{0}+\theta_{1}+\theta_{\infty}}{2}\right)\right],
\end{aligned}
\right.
\end{equation}
which implies the PV equation
\begin{equation}\label{eq:P5}
\frac{\mathrm{d}^{2}y}{\mathrm{d}x^{2}}=\left(\frac{1}{2y}+\frac{1}{y-1}\right)\left(\frac{\mathrm{d}y}{\mathrm{d}x}\right)^{2}
-\frac{1}{x}\frac{\mathrm{d}y}{\mathrm{d}x}+\frac{(y-1)^{2}}{x^{2}}\left(\alpha y+\frac{\beta}{y}\right)
+\frac{\gamma y}{x}-\frac{y(y+1)}{2(y-1)},
\end{equation}
where
$$
\alpha=\frac{1}{8}(\theta_{0}-\theta_{1}+\theta_{\infty})^{2},\ \ \ \ \beta=-\frac{1}{8}(\theta_{0}-\theta_{1}-\theta_{\infty})^{2},\ \ \ \ \gamma=1-\theta_{0}-\theta_{1}.
$$

There exist unique solutions $\Phi_k(z,x)$ of the equation \eqref{LaxPairP5}, which satisfies the  normalized asymptotic behavior as $z\to \infty$
\begin{equation}\label{eq:CHE-infty}\Phi_k(z,x)=(I+O(1/z)) e^{\frac{z}{2}\sigma_3}z^{-\frac{1}{2}\theta_{\infty}\sigma_3}, \quad z\in \Omega_k, \end{equation}
where $\Omega_k$, $k=1,2$, are the Stokes sectors defined in \eqref{eq:CHEStokesSec}. The solutions are  related to each other by the Stokes matrices
\begin{equation}\label{eq:PV-Stokes1}\Phi_2(z)=\Phi_1(z)\hat{S}_1, \quad \hat{S}_1=\begin{pmatrix}
    1 & 0 \\
    \hat{s}_1 & 1
\end{pmatrix}, \end{equation}
\begin{equation}\label{eq:PV-Stokes2}\Phi_1(ze^{2\pi i})e^{\pi i\theta_{\infty}\sigma_3}=\Phi_2(z)\hat{S}_2, \quad \hat{S}_2=\begin{pmatrix}
    1 & \hat{s}_2 \\
  0 & 1
\end{pmatrix}. \end{equation}
Near the regular singular points $ z_0=0$ and $z_1=a$, $\Phi_1(z,x)$ has the asymptotic behaviors
\begin{equation*}
 \Phi_1(z,x) =\Phi_1^{(k)}(x)\left(I+O\left(z-z_k\right)\right)
(z-z_k)^{\frac{1}{2}\theta_k\sigma_3}\hat{E}_k,~~z\to z_k,
\end{equation*}
with  some invertible constant matrices $\hat{E}_k$ and $z$-independent matrices $\Phi_1^{(k)}(x)$, $k=0,1$.
The monodromy data of the system  \eqref{LaxPairP5} is defined as the set
\begin{equation}\label{eq:PVMonDat}\left\{e^{\pi i\theta_{\infty}},e^{\pi i\theta_k}; \hat{S}_1, \hat{S}_2; \hat{E}_0, \hat{E}_1\right\} .\end{equation}
They satisfy the cyclic  condition
 \begin{equation}\label{eq:CHE-CyclicCon}\hat{M}_1\hat{M}_0=\hat{S}_1\hat{S}_2e^{-\pi i\theta_{\infty}\sigma_3}, \quad \hat{M}_k= \hat{E}_k^{-1}e^{\pi i\theta_k\sigma_3} \hat{E}_k, \quad k=0,1. \end{equation}
%

\subsection{Reduction of the  linear systems   for PV to CHE}\label{sec:CHEPV}

In this subsection, we will derive CHE from the linear systems \eqref{Phi1} and \eqref{Phi2} for PV at the  poles, zeros and 1-points (that is, $y(a)=1$) of the solutions to PV.

Substituting the elements of $A(z,x)$ into \eqref{Phi1} gives
\begin{equation}\label{Phi1P5}
\frac{\mathrm{d}^{2}\Phi_{1}}{\mathrm{d}z^{2}}-\left(\frac{1}{z-\lambda_{1}}-\frac{1}{z}-\frac{1}{z-x}\right)\frac{\mathrm{d}\Phi_{1}}{\mathrm{d}z}-P(z,x)\Phi_{1}=0,
\end{equation}
where
$$
\left\{
\begin{aligned}
&\lambda_{1}=\lambda_{1}(x)=\frac{x}{1-\frac{y(v+\frac{1}{2}(\theta_{0}-\theta_{1}+\theta_{\infty}))}{v+\theta_{0}}},\\
&P(z,x)=\frac{1}{4}+\frac{\theta_{0}^{2}}{4z^{2}}+\frac{M(x)}{z(z-x)}+\frac{\theta_{1}^{2}}{4(z-x)^{2}}+\frac{2-\theta_{\infty}}{2(z-x)}\\
&\ \ \ \ \ \ \ \ \ \ \ \ -\frac{1}{z-\lambda_{1}}\left(\frac{1}{2}+\frac{2v+\theta_{0}}{2z}-\frac{2v+\theta_{0}+\theta_{\infty}}{2(z-x)}\right),\\
&M(x)=yv\left(v+\frac{\theta_{0}-\theta_{1}+\theta_{\infty}}{2}\right)+\frac{1}{y}(v+\theta_{0})\left(v+\frac{\theta_{0}+\theta_{1}+\theta_{\infty}}{2}\right)\\
&\ \ \ \ \ \ \ \ \ \ -2\left(v+\frac{\theta_{0}}{2}\right)\left(v+\frac{\theta_{0}+\theta_{\infty}}{2}\right)-\left(v+\frac{1+\theta_{0}}{2}\right)x-\frac{\theta_{\infty}}{2}.
\end{aligned}
\right.
$$
In virtue of the definition of the PV $\tau$-function (see \cite[(C.43)]{Jim2}):
\begin{align}\label{TauP5}
x\frac{\mathrm{d}}{\mathrm{d}x}\log \tau(x)=&-\left[v-\frac{1}{y}\left(v+\frac{\theta_{0}+\theta_{1}+\theta_{\infty}}{2}\right)\right]\left[v+\theta_{0}-y\left(v+\frac{\theta_{0}-\theta_{1}+\theta_{\infty}}{2}\right)\right]\nonumber\\
&-\left(v+\frac{\theta_{0}+\theta_{\infty}}{2}\right)x,
\end{align}
the expression of $M(x)$ is simplified to
$$
M(x)=x\frac{\mathrm{d}}{\mathrm{d}x}\log\tau(x)+\frac{(\theta_{\infty}-1)x}{2}-\frac{\theta_{0}^{2}+\theta_{1}^{2}-\theta_{\infty}^{2}}{4}-\frac{\theta_{\infty}}{2}.
$$
Similarly, substituting the elements of $A(z,x)$ into  \eqref{Phi2} gives
\begin{equation}\label{Phi2P5}
\frac{\mathrm{d}^{2}\Phi_{2}}{\mathrm{d}z^{2}}-\left(\frac{1}{z-\lambda_{2}}-\frac{1}{z}-\frac{1}{z-x}\right)\frac{\mathrm{d}\Phi_{2}}{\mathrm{d}z}-Q(z,x)\Phi_{2}=0,
\end{equation}
where
$$
\left\{
\begin{aligned}
&\lambda_{2}=\lambda_{2}(x)=\frac{x}{1-\frac{v+\frac{1}{2}(\theta_{0}+\theta_{1}+\theta_{\infty})}{yv}},\\
&Q(z,x)=\frac{1}{4}+\frac{\theta_{0}^{2}}{4z^{2}}+\frac{N(x)}{z(z-x)}+\frac{\theta_{1}^{2}}{4(z-x)^{2}}-\frac{2+\theta_{\infty}}{2(z-x)}\\
&\ \ \ \ \ \ \ \ \ \ +\frac{1}{z-\lambda_{2}}\left(\frac{1}{2}+\frac{2v+\theta_{0}}{2z}-\frac{2v+\theta_{0}+\theta_{\infty}}{2(z-x)}\right),\\
&N(x)=x\frac{\mathrm{d}}{\mathrm{d}x}\log\tau(x)+\frac{(\theta_{\infty}+1)x}{2}-\frac{\theta_{0}^{2}+\theta_{1}^{2}-\theta_{\infty}^{2}}{4}+\frac{\theta_{\infty}}{2}.
\end{aligned}
\right.
$$

To eliminate the singularities $z=\lambda_{1}$ and $z=\lambda_{2}$ and obtain CHE, we need
\begin{align}\label{Cod1}
&\lambda_{1}(x)=0,\ \ \mathrm{or} \ \ \lambda_{1}(x)=x,\ \ \mathrm{or}\ \ \lambda_{1}(x)=\infty,\\
&\lambda_{2}(x)=0,\ \ \mathrm{or} \ \ \lambda_{2}(x)=x,\ \ \mathrm{or}\ \ \lambda_{2}(x)=\infty,  \label{Cod2}
\end{align}
in \eqref{Phi1P5} and  \eqref{Phi2P5}, respectively. For this purpose, we consider the
poles, zeros and 1-points of the solution $y$.

According to \cite[(36.4)-(36.7)]{GLS}, the solution $y$ admits the Laurent expansion  near a pole
$$
y(x)=\left\{\begin{aligned}
  &\varepsilon(x-a)^{-1}+b+O(x-a),\ \   &\alpha\neq0,\\
  &ab(x-a)^{-2}+b(x-a)^{-1}+O(1), \ \   &\alpha=0,
\end{aligned} \right.
$$
and the following Taylor expansions near a zero and near a 1-point
$$
y(x)=\left\{\begin{aligned}
  &\delta(x-a)+b(x-a)^{2}+O((x-a)^{3}),\ \   &\beta\neq0,\\
  &b(x-a)^{2}+O((x-a)^{3}), \ \              &\beta=0,
\end{aligned} \right.
$$
$$
y(x)=1+\omega(x-a)+\bigg(\frac{1}{2}+\frac{\omega-1+\theta_{0}+\theta_{1}}{2a}\bigg)(x-a)^{2}+O((x-a)^{3}),
$$
where $b$ is arbitrary and the parameters
\begin{equation}\label{coeffP5}
\left\{
\begin{aligned}
\varepsilon&=\varepsilon_{\pm}=\pm 2a/(\theta_{0}-\theta_{1}+\theta_{\infty}),\\
\delta&=\delta_{\pm}=\pm (\theta_{0}-\theta_{1}-\theta_{\infty})/2a,\\
\omega&=\omega_{\pm}=\pm1.
\end{aligned}\right.
\end{equation}
From the first equation in the compatibility condition \eqref{CCP5} and above expansions, we derive the behaviors of $v,\lambda_{1},\lambda_{2}$ at each  pole, zero or 1-point of $y$, as given in Table 1.

\renewcommand\arraystretch{1.2}
\begin{table}[h]
\begin{center}
\begin{tabular}{|c|c|c|c|c|}
\hline
 Case &   $y$ &  $v$ &  $\lambda_{1}$ &  $\lambda_{2}$ \\
 \hline
$\varepsilon=\varepsilon_{+}$  &~~simple pole~~  & $0$                                                     & 0     &   $O(a)$  \\
\hline
$\varepsilon=\varepsilon_{-}$  & simple pole  & $-\frac{1}{2}(\theta_{0}-\theta_{1}+\theta_{\infty})$   &    $ ~~~~O(a) ~~~~$       &  $a$  \\
\hline
$\alpha=0$                     & double pole  & $0$                                                     &  0      &  $a$  \\
\hline
$\delta=\delta_{+}$            & simple zero  & $-\theta_{0}$                                           &     $O(a)$     &  0  \\
\hline
$\delta=\delta_{-}$            & simple zero  & $-\frac{1}{2}(\theta_{0}+\theta_{1}+\theta_{\infty})$   &  $a$      &    $O(a)$   \\
\hline
$\beta=0$                      & double zero  & $-\frac{1}{2}(\theta_{0}+\theta_{1}+\theta_{\infty})$  &  $a$      &  0  \\
\hline
$\omega=\omega_{+}$            & 1-point      & $O(1)$                                                  &     $O(a)$      &   $ ~~~~O(a) ~~~~$   \\
\hline
$\omega=\omega_{-}$            & 1-point      & $a(x-a)^{-2}+O(1)$                                      & $\infty$  & $\infty$  \\
\hline
\end{tabular}
\end{center}
\caption{Behaviors of $y,v,\lambda_{1},\lambda_{2}$ at the critical point $x=a$}
\label{TableBehavior-y}
\end{table}

It is readily seen from Table 1 that the condition \eqref{Cod1} or  \eqref{Cod2} is fulfilled when $a$ is a pole, zero  or 1-point of $y$ with $y'(a)=\omega_{-}$.  Thus, the equation \eqref{Phi1P5} or \eqref{Phi2P5} is  equivalent to CHE in these cases. For example,  we consider the case $\varepsilon=\varepsilon_{+}$. According to Table 1, we have $v(a)=0 $, $\lambda_{1}(a)=0$. Thus, by taking $x\rightarrow a$ in  \eqref{Phi1P5} and setting
 \begin{equation*}
\Phi_{1}=z^{\frac{\theta_{0}}{2}}(z-a)^{\frac{\theta_{1}}{2}}e^{\frac{1}{2}z}w,
\end{equation*} 
we obtain the CHE
\begin{equation}\label{eq:CHE}
\left\{\begin{aligned}
&\frac{\mathrm{d}^{2}w}{\mathrm{d}z^{2}}+\left(1+\frac{\theta_{0}}{z}+\frac{1+\theta_{1}}{z-a}\right)\frac{\mathrm{d}w}{\mathrm{d}z}+\frac{pz-q}{z(z-a)}w=0,\\
&p=\frac{1}{2}(\theta_{0}+\theta_{1}+\theta_{\infty}),\\
&q=\lim_{x\rightarrow a}x\frac{\mathrm{d}}{\mathrm{d}x}\log\tau(x)
+\frac{(\theta_{0}+\theta_{\infty})a}{2}-\frac{(\theta_{0}+\theta_{1})^{2}
-\theta_{\infty}^{2}}{4}\\
&~~=\frac{a}{4}(\theta_{0}-\theta_{1}+\theta_{\infty})
-\frac{b}{4}(\theta_{0}-\theta_{1}+\theta_{\infty})^{2}
+\frac{1}{4}(\theta_{\infty}+1)(\theta_{0}-\theta_{1}+\theta_{\infty})
-\theta_{0}\theta_{1}.
\end{aligned}\right.
\end{equation}

Actually, we obtain six different CHEs with the characteristic exponents at the singular points  shown in Table 2 and the accessory parameters given below:
\begin{equation}\label{qP5-1}
q=\lim_{x\rightarrow a}x\frac{\mathrm{d}}{\mathrm{d}x}\log\tau(x)+\frac{(\theta_{0}+\theta_{\infty})a}{2}-\frac{(\theta_{0}+\theta_{1})^{2}-\theta_{\infty}^{2}}{4},
\end{equation}
for $\varepsilon=\varepsilon_{+}$, $\delta=\delta_{+}$ and $\delta=\delta_{-}$,
\begin{equation}\label{qP5-2}
q=\lim_{x\rightarrow a}x\frac{\mathrm{d}}{\mathrm{d}x}\log\tau(x)+\frac{(\theta_{0}+\theta_{\infty}+2)a}{2}-\frac{(\theta_{0}+\theta_{1})^{2}-\theta_{\infty}^{2}}{4},
\end{equation}
for $\varepsilon=\varepsilon_{-}$,
\begin{equation}\label{qP5-3}
q=\lim_{x\rightarrow a}\left(x\frac{\mathrm{d}}{\mathrm{d}x}\log\tau(x)-\frac{a}{x-a}\right)+\frac{(\theta_{0}+\theta_{\infty})a}{2}-\frac{(\theta_{0}+\theta_{1})^{2}-\theta_{\infty}^{2}}{4}+1-\theta_{0}-\theta_{1},
\end{equation}
for $\omega=\omega_{-}$ and employing $\Phi_1$,
\begin{equation}\label{qP5-4}
q=\lim_{x\rightarrow a}\left(x\frac{\mathrm{d}}{\mathrm{d}x}\log\tau(x)-\frac{a}{x-a}\right)+\frac{(\theta_{0}+\theta_{\infty}+2)a}{2}-\frac{(\theta_{0}+\theta_{1})^{2}-\theta_{\infty}^{2}}{4}+1-\theta_{0}-\theta_{1},
\end{equation}
for $\omega=\omega_{-}$ and using $\Phi_2$.

\begin{table}[h]
\begin{center}
\begin{tabular}{|c|c|c|c|c|c|}
\hline
 Case &  Equation &   $0$ &  $a$ &   $\infty$ &  $p$ \\
 \hline
$\varepsilon=\varepsilon_{+}$  &  \eqref{Phi1P5}    &  $\theta_{0}$     &  $~~1+\theta_{1}~~$   &  ~~1~~  &   $\frac{1}{2}(\theta_{0}+\theta_{1}+\theta_{\infty})$    \\
\hline
$\varepsilon=\varepsilon_{-}$  & \eqref{Phi2P5}      &  ~~$1+\theta_{0}$~~   &  $\theta_{1}$     &  1  &   $\frac{1}{2}(\theta_{0}+\theta_{1}+\theta_{\infty}+2)$  \\
\hline
$\delta=\delta_{+}$            & \eqref{Phi2P5}   &  $\theta_{0}$     &  $1+\theta_{1}$   &  1  &   $\frac{1}{2}(\theta_{0}+\theta_{1}+\theta_{\infty}+2)$  \\
\hline
$\delta=\delta_{-}$            &  \eqref{Phi1P5}   &  $1+\theta_{0}$   &  $\theta_{1}$     &  1  &   $\frac{1}{2}(\theta_{0}+\theta_{1}+\theta_{\infty})$    \\
\hline
$\omega=\omega_{-}$            &  \eqref{Phi1P5}    &  $1+\theta_{0}$   &  $1+\theta_{1}$   &  1  &   $\frac{1}{2}(\theta_{0}+\theta_{1}+\theta_{\infty}+2)$   \\
\hline
$\omega=\omega_{-}$            &  \eqref{Phi2P5}   &  $1+\theta_{0}$   &  $1+\theta_{1}$   &  1  &   $\frac{1}{2}(\theta_{0}+\theta_{1}+\theta_{\infty}+2)$   \\
\hline
\end{tabular}
\end{center}
\caption{The characteristic exponents of CHEs obtained from equations \eqref{Phi1P5} and \eqref{Phi2P5}    }

\end{table}

We mention that the case $\lambda_{1}(x)=x$ has been considered in \cite{CC,CN1}. According to Table 1, the condition $\lambda_{1}(x)=x$ is equivalent to the case that $y$ has simple zeros with coefficient $\delta=\delta_{-}$ or $y$ has double zeros.

%
%

\subsection{Isomonodromy set of  accessory parameters of CHE}

We  have shown that  the CHE  can be obtained  as a limit of the first row or second row of  the isomonodromy family of  $\Phi(z,x)$  when  $x\to a$, $a$ being one of the poles, zeros or $1$-points of  the  Painlev\'e V transcendents corresponding to the same monodromy data of CHE.  While  the accessory parameters $(a,q)$ are expressed in terms of the parameters in the Laurent or Taylor expansion  of the  PV transcendents. As a consequence, we have the following description of the isomonodromy set of  accessory parameters of CHE.

\begin{thm}\label{thm:IsoSetCHE}
There is a discrete set of pairs of accessory parameters $(a_n,q_n)$  such that  the CHE \eqref{int:CHE}
corresponding to these parameters and the fixed parameters $\gamma$, $\delta$ and $p$, has the same monodromy data  \eqref{eq:CHEMonDat} as the original
one with the parameters $a$ and $q$. Under the relation given in the last equation of  \eqref{eq:CHE},  this set coincides with the set of parameters $(a_n, b_n)$ in the  Laurent expansion near the poles $a_n$ of the unique solution  of PV  \eqref{eq:P5} with the parameters $\theta_0=\gamma$, $\theta_1=\delta-1$, $\theta_{\infty}=2p-\gamma-\delta+1$ and corresponding to the same monodromy data \eqref{eq:CHEMonDat} as the CHE.
\end{thm}

\section{ Accessory parameters of HE }\label{sec:HE}
\subsection{Reduction of the linear system  for PVI   to HE }\label{sec: HE-red}
In this subsection, we will derive HE from the linear system \eqref{Phi1} for PVI  at  the poles, zeros, 1-points ($y(x)=1$) and fixed points  ($y(x)=x$) of the solutions of PVI.

Consider the following Lax pair for PVI equation (see \cite[(C.46)-(C.47)]{Jim2}):
\begin{equation}
\left\{\begin{aligned}\label{eq:P6}
&\frac{\partial\Phi(z,x)}{\partial z}=A(z,x)\Phi(z,x), \ \ A(z,x)=\frac{A_{0}}{z}+\frac{A_{1}}{z-1}+\frac{A_{2}}{z-x},\\
&\frac{\partial\Phi(z,x)}{\partial x}=B(z,x)\Phi(z,x), \ \ B(z,x)=-\frac{A_{2}}{z-x},
\end{aligned}\right.
\end{equation}
where
$$
A_{i}=
\begin{pmatrix}
     v_{i}+\theta_{i} & -u_{i} v_{i} \\
     u_{i}^{-1}\left(v_{i}+\theta_{i}\right) & -v_{i}
\end{pmatrix},\ \ \ i=0,1,2.
$$
If we set
\begin{equation}\label{ki}
\left\{
\begin{aligned}
&\kappa_{1}=-\frac{1}{2}\left(\theta_{0}+\theta_{1}+\theta_{2}-\theta_{\infty}\right),\\
&\kappa_{2}=-\frac{1}{2}\left(\theta_{0}+\theta_{1}+\theta_{2}+\theta_{\infty}\right),\\
&A_{0}+A_{1}+A_{2}=
-\begin{pmatrix}
\kappa_{1} & 0 \\
0 & \kappa_{2}
\end{pmatrix},\\
&A(z,x)_{12}=\frac{-u_{0} v_{0}}{z}+\frac{-u_{1} v_{1}}{z-1}+\frac{-u_{2} v_{2}}{z-x}=\frac{k(z-y)}{z(z-1)(z-x)},\\
&v=\frac{v_{0}+\theta_{0}}{y}+\frac{v_{1}+\theta_{1}}{y-1}+\frac{v_{2}+\theta_{2}}{y-x},
\end{aligned}
\right.
\end{equation}
then the compatibility condition of the above Lax pair gives
\begin{equation}\label{CCP6}
\left\{
\begin{aligned}
\frac{\mathrm{d} y}{\mathrm{d} x}&= \frac{y(y-1)(y-x)}{x(x-1)}\left(2 v-\frac{\theta_{0}}{y}-\frac{\theta_{1}}{y-1}-\frac{\theta_{2}-1}{y-x}\right), \\
\frac{\mathrm{d} v}{\mathrm{d} x}&= \frac{1}{x(x-1)}\bigg\{\left[-3 y^{2}+2(1+x) y-x\right] v^{2}+\big[(2 y-1-x) \theta_{0}+(2 y-x) \theta_{1}\\
&\ \ \ +(2 y-1)\left(\theta_{2}-1\right)\big] v-\kappa_{1}\left(\kappa_{2}+1\right)\bigg\},\\
\frac{\mathrm{d} k}{\mathrm{d} x}&=k\left(\theta_{\infty}-1\right) \frac{y-x}{x(x-1)},\end{aligned}
\right.
\end{equation}
which leads to the PVI equation
\begin{align}\nonumber
\frac{\mathrm{d}^{2} y}{\mathrm{d} x^{2}}=\ & \frac{1}{2}\left(\frac{1}{y}+\frac{1}{y-1}+\frac{1}{y-x}\right)\left(\frac{\mathrm{d} y}{\mathrm{d} x}\right)^{2}-\left(\frac{1}{x}+\frac{1}{x-1}+\frac{1}{y-x}\right) \frac{\mathrm{d} y}{\mathrm{d} x} \label{eq:yP6}\\
&+\frac{y(y-1)(y-x)}{x^{2}(x-1)^{2}}\left(\alpha_0+\beta_0 \frac{x}{y^{2}}+\gamma_0 \frac{x-1}{(y-1)^{2}}+\delta_0 \frac{x(x-1)}{(y-x)^{2}}\right),
\end{align}
where
$$
\alpha_0=\frac{1}{2}\left(\theta_{\infty}-1\right)^{2},\ \ \quad \beta_0=-\frac{1}{2} \theta_{0}^{2},\ \ \quad \gamma_0=\frac{1}{2} \theta_{1}^{2},\ \ \quad \delta_0=\frac{1}{2}\left(1-\theta_{2}^{2}\right).
$$
The PVI $\tau$-function is defined by (see \cite[(C.57)]{Jim2}):
\begin{align}\label{TauP6}
x(x-1)\frac{\mathrm{d}}{\mathrm{d}x}\log\tau(x)=\ &y(y-1)(y-x)\bigg\{v^{2}-\bigg(\frac{\theta_{0}}{y}+\frac{\theta_{1}}{y-1}+\frac{\theta_{2}}{y-x}\bigg)v+\frac{\kappa_{1} \kappa_{2}}{y(y-1)}\bigg\}\nonumber\\
                        & +\theta_{0}\theta_{2}(x-1)+\theta_{1}\theta_{2}x.
\end{align}

Substituting the entries of $A(z,x)$ into \eqref{Phi1} gives
\begin{equation}\label{Phi1P6}
\frac{\mathrm{d}^{2}\Phi_{1}}{\mathrm{d}z^{2}}+P(z,x)\frac{\mathrm{d}\Phi_{1}}{\mathrm{d}z}+Q(z,x)\Phi_{1}=0,
\end{equation}
where
$$
\left\{
\begin{aligned}
&P(z,x)=\frac{1-\theta_{0}}{z}+\frac{1-\theta_{1}}{z-1}+\frac{1-\theta_{2}}{z-x}-\frac{1}{z-y},\\
&Q(z,x)=\frac{\kappa_{1}(\kappa_{2}+1)}{z(z-1)}-\frac{M(x)}{z(z-1)(z-x)}+\frac{y(y-1)v}{z(z-1)(z-y)},\\
&M(x)=x(x-1)\frac{\mathrm{d}}{\mathrm{d}x}\log\tau(x)+\kappa_{1}(y-x)+y(y-1)v-\theta_{0}\theta_{2}(x-1)-\theta_{1}\theta_{2}x.
\end{aligned}
\right.
$$
It is clear that the extra singularity $z=y$ can be removed by  considering the critical values
$$
y(a)=0,\ \ \mathrm{or} \ \ y(a)=1,\ \ \mathrm{or} \ \ y(a)=a,\ \ \mathrm{or} \ \ y(a)=\infty.
$$

First, we consider the critical value $y(a)=a$. By taking the limit $x\rightarrow a$ in  \eqref{Phi1P6}, we obtain the following HE directly
\begin{equation}\label{HE1}
\left\{
\begin{aligned}
&\frac{\mathrm{d}^{2}\Phi_{1}}{\mathrm{d}z^{2}}+\left(\frac{1-\theta_{0}}{z}+\frac{1-\theta_{1}}{z-1}+\frac{-\theta_{2}}{z-a}\right)\frac{\mathrm{d}\Phi_{1}}{\mathrm{d}z}+\frac{pz-q}{z(z-1)(z-a)}\Phi_{1}=0,\\
&p=\frac{1}{4}(\theta_{0}+\theta_{1}+\theta_{2}-\theta_{\infty})(\theta_{0}+\theta_{1}+\theta_{2}+\theta_{\infty}-2),\\
&q=\lim_{x\rightarrow a}x(x-1)\frac{\mathrm{d}}{\mathrm{d}x}\log\tau(x)+\kappa_{1}(\kappa_{2}+1)a-\theta_{0}\theta_{2}(a-1)-\theta_{1}\theta_{2}a.
\end{aligned}
\right.
\end{equation}
The equation \eqref{HE1} has  also been obtained in  \cite{ACP1,ANCC,CN2}. It was also noted therein that the condition $y(a)=a$ can be expressed in terms of the associated $\tau$-function.

Next, we consider the poles of $y$. Let $x=a$ ($a\neq 0,1$) be a movable pole of $y(x)$. Then $y(x)$ possesses the following Laurent expansions (see \cite[ (46.7)]{GLS}):
\begin{equation}\label{yP6Expand}
y(x)=\left\{\begin{aligned}
  &\frac{\varepsilon}{x-a}+b+O(x-a),\ \                        & \alpha_0\neq0,\\
  &\frac{b}{(x-a)^{2}}+\frac{(2a-1)b}{a(a-1)(x-a)}+c_{0}+O(x-a), \ \   & \alpha_0=0,
\end{aligned} \right.
\end{equation}
where $\varepsilon=\varepsilon_{\pm}=\pm a(a-1)/(\theta_{\infty}-1)$, $b$ is arbitrary and
$$
c_{0}=\frac{1}{3}(a+1)+\frac{b}{12a^{2}(a-1)^{2}}\Big\{12a(a-1)+1-a\theta_{0}^{2}+(a-1)\theta_{1}^{2}-a(a-1)(\theta_{2}^{2}-1)\Big\}.
$$

We derive HE by considering each case separately.
\subsection*{Case 1: $a$ is a simple pole with residue $\varepsilon=\varepsilon_{+}$ and $\theta_{\infty}\neq 0$}
In this case, we have from \eqref{TauP6},  \eqref{Phi1P6} and \eqref{yP6Expand} that
\begin{align}\label{behM}
&M(x)=x(x-1)\frac{\mathrm{d}}{\mathrm{d}x}\log\tau(x)+(\theta_{\infty}-1)b-\kappa_{1}a-\theta_{0}\theta_{2}(a-1)-\theta_{1}\theta_{2}a\nonumber \\
&\ \ \ \ \ \ \ \ \ \ +\frac{1}{2}\big[-\theta_{0}+(\theta_{2}-1)(a-1)-2a+1-(\theta_{\infty}-1)a\big]+O(x-a),
\end{align}
and
\begin{align}\label{behtau}
&x(x-1)\frac{\mathrm{d}}{\mathrm{d}x}\log\tau(x)=\theta_{\infty}(1-\theta_{\infty})b\nonumber\\
&\ \ \ \ \ \ \ \ \ \ -\frac{1}{2}\theta_{\infty}\big[-\theta_{0}+(\theta_{2}-1)(a-1)-2a+1-(\theta_{\infty}-1)a\big]-\kappa_{1}^{2}(a-1)\nonumber\\
&\ \ \ \ \ \ \ \ \ \ -\kappa_{1}\kappa_{2}a-\kappa_{1}\big[a\theta_{0}+(a-1)\theta_{1}\big]+\theta_{0}\theta_{2}(a-1)+\theta_{1}\theta_{2}a+O(x-a).
\end{align}
This, together with $\theta_{\infty}\neq 0$, implies that
\begin{align*}
M(x)&=(1-\theta_{\infty}^{-1})x(x-1)\frac{\mathrm{d}}{\mathrm{d}x}\log\tau(x)-\kappa_{1}a-\theta_{\infty}^{-1}\Big\{\kappa_{1}^{2}(a-1)+\kappa_{1}\kappa_{2}a+\kappa_{1}\big[a\theta_{0}\\
&\ \ \ \ +(a-1)\theta_{1}\big]\Big\}-(1-\theta_{\infty}^{-1})\big[\theta_{0}\theta_{2}(a-1)+\theta_{1}\theta_{2}a\big]+O(x-a).
\end{align*}
Taking $x\rightarrow a$ in \eqref{Phi1P6}, we obtain the HE
\begin{equation}\label{HE2}
\left\{
\begin{aligned}
&\frac{\mathrm{d}^{2}\Phi_{1}}{\mathrm{d}z^{2}}+\left(\frac{1-\theta_{0}}{z}+\frac{1-\theta_{1}}{z-1}+\frac{1-\theta_{2}}{z-a}\right)\frac{\mathrm{d}\Phi_{1}}{\mathrm{d}z}
+\frac{pz-q}{z(z-1)(z-a)}\Phi_{1}=0,\\
&p=\frac{1}{4}(\theta_{0}+\theta_{1}+\theta_{2}-\theta_{\infty})(\theta_{0}+\theta_{1}+\theta_{2}+\theta_{\infty}-4),\\
&q=\lim_{x\rightarrow a}(1-\theta_{\infty}^{-1})x(x-1)\frac{\mathrm{d}}{\mathrm{d}x}\log\tau(x)+\kappa_{1}(\kappa_{2}+2)a-\theta_{\infty}^{-1}\Big\{\kappa_{1}^{2}(a-1)\\
&\ \ \ \ \ +\kappa_{1}\kappa_{2}a+\kappa_{1}\big[a\theta_{0}+(a-1)\theta_{1}\big]\Big\}-(1-\theta_{\infty}^{-1})\big[\theta_{0}\theta_{2}(a-1)+\theta_{1}\theta_{2}a\big],
\end{aligned}
\right.
\end{equation}
where $\kappa_1$ and $\kappa_2$ are given in \eqref{ki}.

\subsection*{Case 2:   $a$ is a simple pole with residue $\varepsilon=\varepsilon_{+}$ and $\theta_{\infty}=0$ }
When $\theta_{\infty}=0$, it is seen from \eqref{behM} and \eqref{behtau} that
\begin{align*}
&M(x)=x(x-1)\frac{\mathrm{d}}{\mathrm{d}x}\log\tau(x)-b+\frac{1}{2}\big[-\theta_{0}+(\theta_{2}-1)(a-1)-a\big]-\kappa_{1}a\\
&\ \ \ \ \ \ \ \ \ \ -\theta_{0}\theta_{2}(a-1)-\theta_{1}\theta_{2}a+O(x-a),
\end{align*}
and
\begin{align*}
x(x-1)\frac{\mathrm{d}}{\mathrm{d}x}\log\tau(x)=&-\kappa_{1}^{2}(a-1)-\kappa_{1}\kappa_{2}a-\kappa_{1}\big[a\theta_{0}+(a-1)\theta_{1}\big]\\
& +\theta_{0}\theta_{2}(a-1)+\theta_{1}\theta_{2}a+O(x-a).
\end{align*}
It is straightforward that the free parameter $b$ is independent of $\tau(x)$. To express the accessory parameter in terms of the $\tau$-function as \eqref{HE2}, we need the Schlesinger transformation used by Dubrovin and Kapaev to shift the formal monodromy at infinity $\theta_{\infty}$ by $-2$, i.e. $\theta_{\infty}\mapsto \theta_{\infty}-2$; see \cite[Sec. 2.3.2]{DK}. As a consequence, we obtain the HE
$$
\left\{
\begin{aligned}
&\frac{\mathrm{d}^{2}\Phi_{1}}{\mathrm{d}z^{2}}+\left(\frac{1-\theta_{0}}{z}+\frac{1-\theta_{1}}{z-1}+\frac{1-\theta_{2}}{z-a}\right)\frac{\mathrm{d}\Phi_{1}}{\mathrm{d}z}+\frac{pz-q}{z(z-1)(z-a)}\Phi_{1}=0,\\
&p=\frac{1}{4}(\theta_{0}+\theta_{1}+\theta_{2}+2)(\theta_{0}+\theta_{1}+\theta_{2}-6),\\
&q=\lim_{x\rightarrow a}\frac{3}{2}x(x-1)\frac{\mathrm{d}}{\mathrm{d}x}\log\tilde{\tau}(x)+\tilde{\kappa}_{1}(\tilde{\kappa}_{2}+2)a+\frac{1}{2}\Big\{\tilde{\kappa}_{1}^{2}(a-1)\\
&\ \ \ \ \ +\tilde{\kappa}_{1}\tilde{\kappa}_{2}a+\tilde{\kappa}_{1}\big[a\theta_{0}+(a-1)\theta_{1}\big]\Big\}-\frac{3}{2}\big[\theta_{0}\theta_{2}(a-1)+\theta_{1}\theta_{2}a\big],
\end{aligned}
\right.
$$
where $\tilde{\tau}(x)$ is a new $\tau$-function defined by
\begin{align*}
x(x-1)\frac{\mathrm{d}}{\mathrm{d}x}\log\tilde{\tau}(x)=\ &y(y-1)(y-x)\bigg\{v^{2}-\bigg(\frac{\theta_{0}}{y}+\frac{\theta_{1}}{y-1}+\frac{\theta_{2}}{y-x}\bigg)v+\frac{\tilde{\kappa}_{1} \tilde{\kappa}_{2}}{y(y-1)}\bigg\}\\
                        & +\theta_{0}\theta_{2}(x-1)+\theta_{1}\theta_{2}x
\end{align*}
with
$$
\left\{
\begin{aligned}
&\tilde{\kappa}_{1}=-\frac{1}{2}\left(\theta_{0}+\theta_{1}+\theta_{2}+2\right),\\
&\tilde{\kappa}_{2}=-\frac{1}{2}\left(\theta_{0}+\theta_{1}+\theta_{2}-2\right).\\
\end{aligned}\right.
$$

\subsection*{Case 3:  $a$ is a simple pole with residue $\varepsilon=\varepsilon_{-}$}

From  \eqref{Phi1P6} and \eqref{yP6Expand},   it follows that
\begin{align*}
M(x)&=x(x-1)\frac{\mathrm{d}}{\mathrm{d}x}\log\tau(x)-\frac{a(a-1)}{x-a}+\frac{1}{2}(\theta_{0}+\theta_{1}+2\theta_{2}-4)a\\
&\ \ \ +1-\frac{1}{2}(\theta_{0}+\theta_{2})-\theta_{0}\theta_{2}(a-1)-\theta_{1}\theta_{2}a+O(x-a).
\end{align*}
By taking $x\rightarrow a$ in \eqref{Phi1P6}, we arrive at the HE
\begin{equation}\label{eq: P6-HE}
\left\{
\begin{aligned}
&\frac{\mathrm{d}^{2}\Phi_{1}}{\mathrm{d}z^{2}}+\left(\frac{1-\theta_{0}}{z}+\frac{1-\theta_{1}}{z-1}+\frac{1-\theta_{2}}{z-a}\right)\frac{\mathrm{d}\Phi_{1}}{\mathrm{d}z}+\frac{pz-q}{z(z-1)(z-a)}\Phi_{1}=0,\\
&p=\frac{1}{4}(\theta_{0}+\theta_{1}+\theta_{2}-\theta_{\infty}-2)(\theta_{0}+\theta_{1}+\theta_{2}+\theta_{\infty}-2),\\
&q=\lim_{x\rightarrow a}\bigg(x(x-1)\frac{\mathrm{d}}{\mathrm{d}x}\log\tau(x)-\frac{a(a-1)}{x-a}\bigg)+(\kappa_{1}+1)(\kappa_{2}+1)a\\
&\ \ \ \ \ +\frac{1}{2}(\theta_{0}+\theta_{1}+2\theta_{2}-4)a+1-\frac{1}{2}(\theta_{0}+\theta_{2})-\theta_{0}\theta_{2}(a-1)-\theta_{1}\theta_{2}a\\
&~~=-(1-\theta_{\infty})^2b-\frac{1}{4}
\Big(2(\theta_0+\theta_1)-(\theta_0+\theta_1)^2+\theta_2 ^2-(\theta_{\infty}-2)^2+4(\theta_{\infty}-1)\Big)a\\
&\qquad -\frac{1}{2}(\theta_0+\theta_2)+\theta_{\infty}-1+
\frac{1}{4}\Big((\theta_0+\theta_2)^2-\theta_1 ^2+(\theta_{\infty}-2)^2\Big),
\end{aligned}
\right.
\end{equation}
where $\kappa_1$ and $\kappa_2$ are given in \eqref{ki}.
\subsection*{Case 4: $a$ is a double pole }
We see from  \eqref{Phi1P6} and \eqref{yP6Expand} that
\begin{align*}
M(x)&=x(x-1)\frac{\mathrm{d}}{\mathrm{d}x}\log\tau(x)-\frac{a(a-1)}{x-a}-\frac{1}{2}(\theta_{0}+\theta_{1}-4)a-\frac{1}{2}(\theta_{0}+\theta_{2})\\
&\ \ \ -\theta_{0}\theta_{2}(a-1)-\theta_{1}\theta_{2}a+O(x-a).
\end{align*}
By taking $x\rightarrow a$ in \eqref{Phi1P6}, we arrive at the HE
\begin{equation}\label{HE:double-pole}
\left\{
\begin{aligned}
&\frac{\mathrm{d}^{2}\Phi_{1}}{\mathrm{d}z^{2}}+\left(\frac{1-\theta_{0}}{z}+\frac{1-\theta_{1}}{z-1}+\frac{1-\theta_{2}}{z-a}\right)\frac{\mathrm{d}\Phi_{1}}{\mathrm{d}z}+\frac{pz-q}{z(z-1)(z-a)}\Phi_{1}=0,\\
&p=\frac{1}{4}(\theta_{0}+\theta_{1}+\theta_{2}-1)(\theta_{0}+\theta_{1}+\theta_{2}-3),\\
&q=\lim_{x\rightarrow a}\bigg(x(x-1)\frac{\mathrm{d}}{\mathrm{d}x}\log\tau(x)-\frac{a(a-1)}{x-a}\bigg)+\kappa_{1}(\kappa_{2}+2)a\\
&\ \ \ \ \ -\frac{1}{2}(\theta_{0}+\theta_{1}-4)a-\frac{1}{2}(\theta_{0}+\theta_{2})-\theta_{0}\theta_{2}(a-1)-\theta_{1}\theta_{2}a\\
&~~=\frac{a^2(a-1)^2}{b}+\Big (\kappa_1\left(\kappa_2+\theta_2+3\right )+3\Big  ) a
+\Big  ( \kappa_1\left(\kappa_2+\theta_1+1\right )-1 \Big  ),
\end{aligned}
\right.
\end{equation}
where $\kappa_1$ and $\kappa_2$ are given in \eqref{ki}, with $\theta_\infty=1$ in this case.

It is noted that the above HEs are also obtained from the linear system at the poles of the solutions of PVI in \cite{DK}, while the accessory parameters in these equations are expressed in terms of the free parameter of the Laurent expansion of $y$ at the poles.

Finally, we  consider  the critical values  $y(x)=0$ and $y(x)=1$.
It follows from  \cite[(46.8)-(46.9)]{GLS} that $y$ admits the following Taylor expansions:
\begin{equation}\label{yP6Ser1}
y(x)=\left\{\begin{aligned}
  &\lambda(x-a)+b(x-a)^{2}+O((x-a)^{3}),\ \   & \beta\neq0,\\
  &b(x-a)^{2}+O((x-a)^{3}), \ \              & \beta=0,
\end{aligned} \right.
\end{equation}
\begin{equation}\label{yP6Ser2}
y(x)=\left\{\begin{aligned}
  &1+\omega(x-a)+b(x-a)^{2}+O((x-a)^{3}),\ \   & \gamma\neq0,\\
  &1+b(x-a)^{2}+O((x-a)^{3}), \ \   & \gamma=0,
\end{aligned} \right.
\end{equation}
where $b$ is arbitrary and
$$
\left\{
\begin{aligned}
&\lambda=\lambda_{\pm}=\pm \frac{\theta_{0}}{a-1},\\
&\omega=\omega_{\pm}=\pm \frac{\theta_{1}}{a}.
\end{aligned}
\right.
$$
It is seen from \eqref{Phi1P6}, \eqref{yP6Ser1} and \eqref{yP6Ser2} that
$$
M(x)=x(x-1)\frac{\mathrm{d}}{\mathrm{d}x}\log\tau(x)-\theta_{0}-\kappa_{1}a-\theta_{0}\theta_{2}(a-1)-\theta_{1}\theta_{2}a+O(x-a),
$$
for $\lambda=\lambda_{+}$,
$$
M(x)=x(x-1)\frac{\mathrm{d}}{\mathrm{d}x}\log\tau(x)-\kappa_{1}a-\theta_{0}\theta_{2}(a-1)-\theta_{1}\theta_{2}a+O(x-a),
$$
for $\lambda=\lambda_{-}$,
$$
M(x)=x(x-1)\frac{\mathrm{d}}{\mathrm{d}x}\log\tau(x)-\frac{\theta_{0}}{2}-\kappa_{1}a-\theta_{0}\theta_{2}(a-1)-\theta_{1}\theta_{2}a+O(x-a),
$$
for $\beta=0$,
$$
M(x)=x(x-1)\frac{\mathrm{d}}{\mathrm{d}x}\log\tau(x)+\kappa_{1}(1-a)-\theta_{0}\theta_{2}(a-1)-\theta_{1}\theta_{2}a+O(x-a),
$$
for $\omega=\omega_{+}$,
$$
M(x)=x(x-1)\frac{\mathrm{d}}{\mathrm{d}x}\log\tau(x)+\theta_{1}+\kappa_{1}(1-a)-\theta_{0}\theta_{2}(a-1)-\theta_{1}\theta_{2}a+O(x-a),
$$
for $\omega=\omega_{-}$,
$$
M(x)=x(x-1)\frac{\mathrm{d}}{\mathrm{d}x}\log\tau(x)+\frac{\theta_{1}}{2}+\kappa_{1}(1-a)-\theta_{0}\theta_{2}(a-1)-\theta_{1}\theta_{2}a+O(x-a),
$$
for $\gamma=0$.\\

Taking $x\rightarrow a$ in \eqref{Phi1P6}, we then obtain six  HEs. We show the exponent parameters at the singularities 0, 1, $a$ and $\infty $ in Table 3 and list the accessory parameters  below:
\begin{align*}
q=\lim_{x\rightarrow a}x(x-1)\frac{\mathrm{d}}{\mathrm{d}x}\log\tau(x)+\theta_{0}(\theta_{2}-1)+\left(\kappa_{1}\kappa_{2}+\theta_{0}
-\theta_{0}\theta_{1}-\theta_{1}\theta_{2}-2\theta_{0}\theta_{2}\right) a,
\end{align*}
for $\lambda=\lambda_{+}$,
$$
q=\lim_{x\rightarrow a}x(x-1)\frac{\mathrm{d}}{\mathrm{d}x}\log\tau(x)+\theta_{0}\theta_{2}+(\kappa_{1}\kappa_{2}-\theta_{0}\theta_{2}-\theta_{1}\theta_{2})a,
$$
for $\lambda=\lambda_{-}$,
$$
q=\lim_{x\rightarrow a}x(x-1)\frac{\mathrm{d}}{\mathrm{d}x}\log\tau(x)+(\kappa_{1}\kappa_{2}-\theta_{1}\theta_{2})a,
$$
for $\beta=0$,
$$
q=\lim_{x\rightarrow a}x(x-1)\frac{\mathrm{d}}{\mathrm{d}x}\log\tau(x)+\kappa_{1}+\theta_{0}\theta_{2}+(\kappa_{1}\kappa_{2}-\theta_{0}\theta_{2}-\theta_{1}\theta_{2})a,
$$
for $\omega=\omega_{+}$,
$$
q=\lim_{x\rightarrow a}x(x-1)\frac{\mathrm{d}}{\mathrm{d}x}\log\tau(x)+\theta_{1}+\theta_{0}\theta_{2}+\kappa_{1}+(\kappa_{1}\kappa_{2}+\theta_{1}-\theta_{0}\theta_{1}-\theta_{0}\theta_{2}-\theta_{1}\theta_{2})a,
$$
for $\omega=\omega_{-}$,
$$
q=\lim_{x\rightarrow a}x(x-1)\frac{\mathrm{d}}{\mathrm{d}x}\log\tau(x)+\kappa_{1}+\theta_{0}\theta_{2}+(\kappa_{1}\kappa_{2}-\theta_{0}\theta_{2})a,
$$
for $\gamma=0$.
\\

\begin{table}[h]
\centering
\begin{tabular}{|c|c|c|c|c|}
\hline
Case & 0 & 1 & $a$ & $\infty$ \\
\hline
$\lambda=\lambda_{+}$    & $\theta_{0}$         & $1-\theta_{1}$       & $1-\theta_{2}$       & $\frac{1}{4}(\theta_{0}-\theta_{1}-\theta_{2}+\theta_{\infty})(\theta_{0}-\theta_{1}-\theta_{2}-\theta_{\infty}+2)$ \\
\hline
$\lambda=\lambda_{-}$    & $-\theta_{0}$        & $1-\theta_{1}$       & $1-\theta_{2}$       & $\frac{1}{4}(\theta_{0}+\theta_{1}+\theta_{2}-\theta_{\infty})(\theta_{0}+\theta_{1}+\theta_{2}+\theta_{\infty}-2)$      \\
\hline
$\beta=0$                & $0$                  & $1-\theta_{1}$       & $1-\theta_{2}$       & $\frac{1}{4}(\theta_{1}+\theta_{2}-\theta_{\infty})(\theta_{1}+\theta_{2}+\theta_{\infty}-2)$      \\
\hline
$\omega=\omega_{+}$      & $1-\theta_{0}$       & $-\theta_{1}$        & $1-\theta_{2}$       & $\frac{1}{4}(\theta_{0}+\theta_{1}+\theta_{2}-\theta_{\infty})(\theta_{0}+\theta_{1}+\theta_{2}+\theta_{\infty}-2)$        \\
\hline
$\omega=\omega_{-}$      & $1-\theta_{0}$       & $\theta_{1}$        & $1-\theta_{2}$        & $\frac{1}{4}(\theta_{1}-\theta_{0}-\theta_{2}+\theta_{\infty})(\theta_{1}-\theta_{0}-\theta_{2}-\theta_{\infty}+2)$    \\
\hline
$\gamma=0$               & $1-\theta_{0}$       & $0$                  & $1-\theta_{2}$       & $\frac{1}{4}(\theta_{0}+\theta_{2}-\theta_{\infty})(\theta_{0}+\theta_{2}+\theta_{\infty}-2)$        \\
\hline
\end{tabular}
\caption{The characteristic exponents at the singularities of HEs}
\end{table}

We mentioned that,  in the cases $\lambda=\lambda_{+}$ and $\omega=\omega_{-}$, use has be made of the transformations $\Phi_{1}=z^{\theta_{0}}w$ and $\Phi_{1}=(z-1)^{\theta_{1}}w$ to get the canonical form of HE \eqref{int:HE}.
%

\subsection{Isomonodromy set of  accessory parameters of HE}
Consider the HE \eqref{int:HE} with  parameters $a$ and $q$.
There exist uniquely two linear independent solutions of  \eqref{int:HE}, namely $(y_1, y_2)$, satisfied the normalized asymptotic behavior as $z\to \infty$:
\begin{equation}\label{eq:HE-infty}\left(y_1(z), y_2(z)\right)\sim \left((1/z)^{\alpha}, (1/z)^{\beta}\right),\end{equation}
where $\alpha$ and $\beta$ are the characteristic exponents at $z=\infty$. We have the asymptotic behaviors of $(y_1, y_2)$ near the singular points $ z_0=0,z_1=1,z_2=a $
\begin{equation*}
 \left(y_1(z), y_2(z)\right) \sim \left((z-z_k)^{\theta_k}, (z-z_k)^{0}\right)E_k,~~k=0,1,2,
\end{equation*}
with the characteristic exponents $\theta_0=1-\gamma$,  $\theta_1=1-\delta$, $\theta_2=1-\epsilon$ and some invertible constant matrix $E_k$, $k=0,1,2$. Under an
 analytic continuation along a closed loop around a singular point,  we obtain another two linear independent solutions of the same equation, which  are therefore related to $(y_1, y_2)$ by
  \begin{equation*}
  \left(y_1\left(z_k+e^{2\pi i}(z-z_k)\right),~ y_2\left(z_k+e^{2\pi i}(z-z_k)\right)\right)=(y_1(z), y_2(z)) e^{\pi i\theta_k}M_k,
  \end{equation*}
\begin{equation*}
 \left(y_1(e^{2\pi i}z),~ y_2(e^{2\pi i}z)\right )=\left(y_1(z), y_2(z)\right) e^{-\pi i(\alpha+\beta)} M_{\infty}.
\end{equation*}
 Here the constant matrices  are known as the monodromy matrices  and determined by the connection matrices
  and the characteristic exponents
 $$  M_{k}=E_k^{-1}e^{\pi i\theta_k\sigma_3} E_k, \quad  M_{\infty}=e^{\pi i\theta_{\infty}\sigma_3}, $$
 with $k=0,1,2$ and $\theta_{\infty}=-\alpha+\beta$.
The monodromy data of  HE  \eqref{int:HE} is then constituted by
\begin{equation}\label{eq:HEMon}\left\{e^{\pi i\theta_{\infty}},e^{\pi i\theta_0},e^{\pi i\theta_1},e^{\pi i\theta_2}; E_0, E_1, E_2\right\},\end{equation}
where the characteristic exponents are related to the fixed parameters in  \eqref{int:HE} by $\theta_0=1-\gamma$,  $\theta_1=1-\delta$, $\theta_2=1-\epsilon$  and $\theta_{\infty}=-\alpha+\beta$.
The monodromy matrices  satisfies the cyclic  condition
 $$M_{\infty}M_2M_1M_0=I.$$  According to \cite{Jim1}, the monodromy matrices can be written explicitly in terms of
the characteristic exponents  and the parameters
$$2\cos(\pi \sigma_{jk})=\Tr M_jM_k=\Tr M_kM_j$$
with $j,k=0,1,2$ and $j<k$. Moreover,  by the cyclic condition, only two of the parameters $\{\sigma_{01}, \sigma_{02}, \sigma_{12}\}$ are independent.   Therefore, we get the same number of independent parameters of the monodromy matrices as the  parameters  $a$ and $q$ in HE equation.

To study the isomonodromy deformation of the  HE equation, it is  convenient to consider the matrix Fuchsian system with
four regular singular points. As shown in Section \ref{sec: HE-red}, the isomonodromy deformation of  \eqref{eq:P6} is described by the PVI equation \eqref{CCP6}. Moreover,
let $\Phi(z,x)$ be the fundamental solution of \eqref{eq:P6} normalized at infinity, then $\Phi(z,x)$ has the asymptotic behaviors near the singular points:
 $$\Phi(z,x)=\Phi_0^{(k)}(x)(I+O(z-z_k))(z-z_k)^{\frac{1}{2}\theta_k}(z-z_k)^{\frac{1}{2}\theta_k\sigma_3} \hat{E}_{k},$$
 and
 $$\Phi(z,x)=(I+O(1/z))z^{\frac{\alpha+\beta}{2}}z^{-\frac{1}{2}\theta_{\infty}\sigma_3},$$
 with $z_0=0$, $z_1=1$, $z_2=x$.  Here the  connection matrix $\hat{E}_{k}$ are certain invertible constant matrices.
Similarly, the  analytic continuation of  $\Phi(z,x)$  along a closed loop around the singular points are related to  $\Phi(z,x)$  by the monodromy matrices
$$\hat{M}_{\infty}=e^{\pi i\theta_{\infty}\sigma_3}, \quad \hat{M}_k=\hat{E}_{k} ^{-1}e^{\pi i \theta_k\sigma_3}\hat{E}_{k},$$
with $k=0,1,2$. While  the HE  can be obtained from a family of isomonodromy deformation system by taking certain limit procedure at the poles of the solution of PVI equation; see for instance \eqref{eq: P6-HE}.  In this way, we obtain a family of accessory parameters sharing the same monodromy data as stated in the following theorem.

\begin{thm}\label{thm:IsoSetHE}
There is a discrete set of pairs of accessory parameters $(a_n,q_n)$  such that  the HE \eqref{int:HE}
corresponding to these parameters has the same monodromy data \eqref{eq:HEMon} as the original
one with the parameters $a$ and $q$. Under the  bijection given in the last equation of  \eqref{eq: P6-HE}, and that of \eqref{HE:double-pole},  this set coincides with the set of parameters $(a_n, b_n)$ in the  Laurent expansion \eqref{yP6Expand} near the poles $a_n$ of the unique solution  of PVI equation \eqref{eq:yP6} with the same monodromy data \eqref{eq:HEMon} as the  HE.
\end{thm}

\begin{proof}
Consider the HE \eqref{int:HE} with given accessory parameters $a$
and $q$. There exist unique fundamental solutions $(y_1, y_2)$ to this equation with the normalized behavior at infinity \eqref{eq:HE-infty}. The solutions are corresponding to the monodromy data  \eqref{eq:HEMon} as mentioned at the beginning of this section. On the other side, there is a unique solution $\Phi(z,x)$ normalized at infinity  of the system  \eqref{eq:P6} with the parameters related to the fixed parameters of  HE \eqref{int:HE} by $\theta_0=1-\gamma$,  $\theta_1=1-\delta$, $\theta_2=1-\epsilon$  and $\theta_{\infty}=-\alpha+\beta$. Here the $x$-dependence of $\Phi(z,x)$ is described by the unique solution of the PVI equation with
the Laurent expansion as given in \eqref{yP6Expand} with the parameter $b$ given by the last equation of  \eqref{eq: P6-HE}  and the parameter $\theta_{\infty} \neq1$ therein.
When $\theta_{\infty}=1$, we consider the  Laurent expansion with a double pole at $a$ in \eqref{yP6Expand}. The solution $\Phi(z, x)$ is corresponding to  the monodromy data of the form $\left\{e^{\pi i\theta_{\infty}}, e^{\pi i\theta_0},e^{\pi i\theta_1},e^{\pi i\theta_2}; \hat{E}_0, \hat{E}_1, \hat{E}_2\right\}$ 
which is independent of $x$.

Using the limiting procedure shown in \eqref{eq: P6-HE},  we obtain the HE  \eqref{int:HE} with accessory parameters $a$  and $q$ as limit of the first row of the isomonodromy family $\Phi(z;x)$ as $x\to a$.  It follows from the $x$-independent of the monodromy data that $E_k=\hat{E}_k$, $k=0,1,2$. Thus, we have shown that any given accessory parameters $(a,q)$ is related by the last equation of  \eqref{eq: P6-HE} to the pole parameters $(a,b)$ of the unique solution of PVI corresponding to the same monodromy data \eqref{eq:HEMon} when  the parameter $\theta_{\infty} \neq1$. 
In the case $\theta_{\infty}=1$, we consider \eqref{HE:double-pole} and similar analysis applies.
By the meromorphic property of the PVI solution, the poles of each PVI solution are discrete and hence the set of pairs of accessory parameters of HE sharing the same monodromy data. We complete the proof of
Theorem \ref{thm:IsoSetHE}.
\end{proof}


\section{Asymptotics of the accessory parameters}\label{Asyofq}
In this section, we will derive some asymptotic approximations for the
 isomonodromy sets of accessory parameters  of some Heun class equations expressed in terms of the monodromy data. The derivation are based on  the connection between the accessory parameters and the
 the Laurent or Taylor coefficients for the corresponding Painlev\'e functions obtained in the previous sections and the  asymptotic expansions for the Painlev\'{e} transcendents known in the literature.
The equations of Heun class we considered in the section include the RBHE, CHE, and DHE.

\subsection{Asymptotics of the accessory parameter of RBHE}
  Consider the RBHE \eqref{int:RBHE} with the accessory parameters $(a_n,q_n)$  such that the equation
corresponding to these parameters has the monodromy data \eqref{eq:RBHEMonDat} specified by
 \begin{equation}\label{eq:RBHEMonSpe}
s_{-1}=-e^{-2\alpha\pi i}, \ \ \ \ \  s_{0}=e^{-2\beta\pi i}, \ \ \ \ \ s_{1}=-e^{2\alpha\pi i},
\end{equation}
with $i\beta\in \mathbb{R}$ and $\alpha>-1/2$.
 The corresponding Painlev\'e XXXIV transcendents and their associated Lax pair  play important roles in  random matrix theory when a Fisher-Hartwig singularity located at an interior point where the density of the equilibrium density vanishes quadratically \cite{CK}, or at  the right edge of the spectrum where typically the density vanishes like a square root \cite{ik, WXZ}. The asymptotics of the Painlev\'e XXXIV transcendents and their associated $\tau$-functions have been  worked out in \cite{DXZ,ik,WXZ}. Using these asympotics and the relation between the isomonodromy set of accessory parameters and the corresponding Painlev\'e XXXIV transcendents  given in Theorem \ref{thm:IsoSetRBHE}, we derive the asymptotic behavior of the accessory parameters $(a_n,q_n)$ expressed in terms of the given monodromy data
\eqref{eq:RBHEMonSpe}.

{\thm\label{thm:AsyRBHE}{ Let $(a_n,q_n)$ be the sequence of  accessory parameters  such that the RBHE \eqref{int:RBHE} corresponding to these parameters has the same monodromy data \eqref{eq:RBHEMonSpe}, then we have the asymptotic approximations as $n\to\infty$
\begin{align}\label{thm:zero2P34Asy}
\frac{4}{3}|a_{n}|^{\frac{3}{2}}= 2 n\pi+2i\beta\ln 3n\pi+4i\beta\ln2-2\arg\Gamma(\alpha-\beta)+\frac{1}{2}(2\alpha+1)\pi+O\left(\frac{\ln n}{n}\right),
\end{align}
and
\begin{equation}\label{thm:qP34Asy}
q(a_{n})= 2\beta i|a_{n}|^{\frac{1}{2}}-\frac{1}{2}\left(\alpha-\alpha^{2}+3\beta^2\right)
a_{n}^{-1}+O\left(|a_{n}|^{-\frac{5}{2}}\right).
\end{equation}
}}
\begin{proof}
 According to  \cite[Theorem 2]{WXZ} and  \cite[Proposition 3.4]{DXZ},  we have the asymptotic behaviors for the Painlev\'e XXXIV transcendents $y(x)$ and the associated $\frac{\mathrm{d}}{\mathrm{d}x}\log\tau(x)$, corresponding to the  Stokes multipliers \eqref{eq:RBHEMonSpe}, as $x\rightarrow-\infty$:
\begin{equation}\label{yP34Asy}
\begin{aligned}
y(x)&=\frac{2|\alpha-\beta|}{\sqrt{|x|}}\cos\left(\frac{\theta(x)}{2}+\arg\Gamma(1+\alpha-\beta)-\frac{\pi}{4}\right)\\
&\quad  \times\cos\left(\frac{\theta(x)}{2}+\arg\Gamma(\alpha-\beta)+\frac{\pi}{4}\right)+O(x^{-2}),
\end{aligned}
\end{equation}
and
\begin{align}\label{TauP34Asy}
\frac{\mathrm{d}}{\mathrm{d}x}\log\tau(x)&= 2\beta i|x|^{\frac{1}{2}}-\frac{|\alpha-\beta|}{2x}\sin\left(\theta(x)+2\arg\Gamma(\alpha-\beta)+\arg(\alpha-\beta)\right)\nonumber \\
&\quad  +\frac{\alpha^{2}-3\beta^{2}}{2x}+O(|x|^{-\frac{5}{2}}),
\end{align}
where
$$
\theta(x)=\frac{4}{3}|x|^{\frac{3}{2}}-3i\beta\ln |x|-\alpha\pi-6i\beta\ln2,
$$
and $\beta\in i\mathbb{R}$.

It follows from \eqref{yP34Asy} that $y(x)$ admits a sequences of simple zeros  lying on the negative real axis and tending to infinity  with the asymptotic approximation given in \eqref{thm:zero2P34Asy}.
Moreover, the leading coefficient of the Taylor expansion of $y(x)$ near the zero $a_n$ is $ 2\alpha$.
The relation in the second equation of \eqref{RBHE2}, together with  \eqref{thm:zero2P34Asy} and \eqref{TauP34Asy}, then implies \eqref{thm:qP34Asy}.
This completes the proof of Theorem \ref{thm:AsyRBHE}.

\end{proof}

\subsection{Asymptotics of the accessory parameter of CHE}

 In this subsection, we consider the CHE  \eqref{eq:CHE} with the accessory parameters $(a_n,q_n)$  such that the equation corresponding to these parameters has the monodromy data \eqref{eq:CHEMonDat} parameterized in terms of $\sigma$ and $s$ as given in \eqref{eq:CHE-E1}-\eqref{eq:CHE-E0}.
 In the seminal work \cite{Jim1}, Jimbo  derived the  asymptotic expansion for the PV tanscendents and the associated  $\tau$-function corresponding to the monodromy data.
 Using these asympotics and the relation between the isomonodromy set of accessory parameters and the corresponding Painlev\'e V  transcendents  given in Theorem  \ref{thm:IsoSetCHE}, we derive the asymptotic behavior of the accessory parameters $(a_n,q_n)$ expressed in terms of the given monodromy data.

{\thm\label{thm:AsyCHE}{ Let $(a_n,q_n)$ be the sequence of  accessory parameters  such that the CHE   \eqref{eq:CHE}
corresponding to these parameters has the same monodromy data  \eqref{eq:CHEMonDat} parameterized in terms of $\sigma$ and $s$ as given in \eqref{eq:CHE-E1}-\eqref{eq:CHE-E0}, then we have the asymptotic approximations as $n\to\infty$
\begin{equation}\label{thm:P5polean}
 \ln |a_n| \sim \left( \Re \sigma \ln|c_0| +\Im \sigma \arg c_0  -2\pi n |\Im \sigma|\right )/|\sigma|^2,  \end{equation}
and
\begin{equation}\label{thm:qP5Asy1}
q(a_{n})= \frac{\sigma^{2}-(\theta_{0}+\theta_{1})^{2}}{4}-\bigg[\frac{(\theta_{\infty}-1)(\theta_{1}^{2}-(\theta_{0}-1)^{2})}{4(\sigma^{2}-1)}
+\frac{1-2\theta_{0}-\theta_{\infty}}{4}\bigg]a_{n}+O(a_{n}^2),
\end{equation}
where
$c_0= \frac{(\sigma+\theta_{\infty})(\sigma+\theta_{0}+\theta_{1})\Gamma(1+\sigma)^2\Gamma(\frac{1}{2}(\theta_1+\theta_0-\sigma)+1)\Gamma(\frac{1}{2}(\theta_1-\theta_0-\sigma)+1)\Gamma(\frac{1}{2}(\theta_{\infty}-\sigma)+1)}{(\sigma-\theta_{\infty})(\sigma-\theta_{0}-\theta_{1})\Gamma(1-\sigma)^2\Gamma(\frac{1}{2}(\theta_1+\theta_0+\sigma)+1)\Gamma(\frac{1}{2}(\theta_1-\theta_0+\sigma)+1)\Gamma(\frac{1}{2}(\theta_{\infty}+\sigma)+1)}\frac{1}{s}
$
with $\Im \sigma \neq0$ and $s\neq0$.
}}
\begin{proof}
According to \cite{Jim1}, we have the following asymptotic expansion  for the $\tau$-function of PV  as $x\rightarrow0$
\begin{align}\label{TauP5Asy}
\tau(x)\sim\ &\text{const.}\ x^{\frac{1}{4}(\sigma^{2}-\theta_{\infty}^{2})}\nonumber\\
&\times\Bigg\{1-\frac{\theta_{\infty}\left(\sigma^{2}+\theta_{1}^{2}-\theta_{0}^{2}\right)}{4 \sigma^{2}}x -\frac{\left(\sigma-\theta_{\infty}\right)\left[\left(\sigma-\theta_{1}\right)^{2}-\theta_{0}^{2}\right]}{8 \sigma^{2}(1+\sigma)^{2}} x(\rho x^{\sigma})\nonumber\\
&+\frac{\left(\sigma+\theta_{\infty}\right)\left[\left(\sigma+\theta_{1}\right)^{2}-\theta_{0}^{2}\right]}{8 \sigma^{2}(1-\sigma)^{2}} x(\rho x^{\sigma})^{-1}+\sum_{j=2}^{\infty}x^j\sum_{k=-j}^{j}c_{jk}x^{k\sigma}\Bigg\},
\end{align}
\begin{equation}\label{rho}
\rho=\frac{\Gamma(1-\sigma)^2\Gamma(\frac{1}{2}(\theta_1+\theta_0+\sigma)+1)\Gamma(\frac{1}{2}(\theta_1-\theta_0+\sigma)+1)\Gamma(\frac{1}{2}(\theta_{\infty}+\sigma)+1)}{\Gamma(1+\sigma)^2\Gamma(\frac{1}{2}(\theta_1+\theta_0-\sigma)+1)\Gamma(\frac{1}{2}(\theta_1-\theta_0-\sigma)+1)\Gamma(\frac{1}{2}(\theta_{\infty}-\sigma)+1)}s.
\end{equation}
Here  the parameters  $\sigma\neq0$ and $s$ are the parameterization of the monodromy data of the corresponding  CHE
as given in  \eqref{eq:CHE-sigama}-\eqref{eq:CHE-E0}.  

As mentioned in  \cite[Remark 1]{Jim1}, the asymptotics of the solution $y$ of PV can also be obtained from the asymptotic analysis carried out therein. The small-$x$ asymptotic expansion  for $y$ was also derived in \cite{Shun1}.  For
$\Im \sigma\neq0$, we see from the asymptotics of  the PV tanscendents $y$ that it admits a sequence of simple poles $\{a_{n}\}_{n\in\mathbb{N}}$ such that
\begin{equation}\label{an1P5}
a_{n}^{\sigma}\sim c_0, \quad c_0= \frac{(\sigma+\theta_{\infty})(\sigma+\theta_{0}+\theta_{1})}{(\sigma-\theta_{\infty})(\sigma-\theta_{0}-\theta_{1})} \rho^{-1},  \quad \mbox {as}\quad a_{n}\rightarrow 0.
\end{equation}
The sequence of poles are  clustering at zero along the spiral described by
\begin{equation}\label{eq:speral1}
|\sigma|^2 \ln |a_n| - \Re \sigma \ln|c_0| -\Im \sigma \arg c_0 \sim -2\pi n |\Im \sigma|,  \end{equation}
and
\begin{equation}\label{eq:speral2}
|\sigma|^2 \arg a_n -\Re \sigma  \arg c_0+\Im \sigma \ln |c_0|  \sim 2\pi n \Re \sigma; \end{equation}
see for instance  \cite[Theorem 2.8]{Shun1}.
It can be checked that the residues of $y$ at $a_{n}$  equal  $\varepsilon_{+}$  as given in \eqref{coeffP5}.

Thus, from  \eqref{rho} and \eqref{eq:speral1} we have \eqref{thm:P5polean}.
Substituting \eqref{TauP5Asy}, \eqref{an1P5} into \eqref{qP5-1}, we obtain
\eqref{thm:qP5Asy1}.  This completes the proof of Theorem \ref{thm:AsyCHE}.

 \end{proof}

It should be mentioned that similar formulae for the asymptotic approximations of the accessory parameters as $a\to 0$ have also been derived in \cite{CC} by considering the zeros of a special $\tau$-function of PV, with applications in black holes; see \cite[(2.62a)-(2.62c)]{CC}.
In general, the asymptotic expansions of the PV solutions $y(x)$ and the associated $\tau$-functions   near infinity are rather complicated. There exists no general asymptotic expansions for $y(x)$  or $\tau(x)$ as $x\rightarrow \infty$ except along some special rays: $\mathrm{arg}x=0,\pi/2,\pi,3\pi/2$; see \cite{AK, LNR, Shun2}.
In \cite{AK,Shun2}, the asymptotics for $y(x)$ and the logarithmic derivative of  the $\tau$-functions as $x\rightarrow i\infty$ are established.  From the asymptotic expansions, it is shown in \cite{Shun2} that under certain  conditions the PV solutions  admit  sequences of  poles and of  zeros lying on the imaginary axis and tending to $i\infty$.
Combining these results with the expressions of the  accessory parameters of CHE in terms of the $\tau$-function given in \eqref{qP5-1}-\eqref{qP5-2}, the asymptotic approximations of the isomonodromic set of accessory parameters $(a_n,q_n)$ corresponding to some special monodromy data may also be obtained.

\subsection{Asymptotics of the accessory parameter of DHE}
In the pioneering work of  McCoy, Tarcy and Wu \cite{MTW}, the asymptotics and the connection formulae for one-parameter family of solutions to PIII were derived rigorously. These  solutions have important applications in the  analysis of  two-dimensional Ising model \cite{WT}. More precisely, they showed in \cite{MTW} that there are one-parameter family of solutions of PIII with the asymptotics
\begin{equation}\label{yP3AsyInfty}
y(x; \nu,\lambda)\sim 1-\lambda\Gamma\left(\nu+\frac{1}{2}\right)2^{-2\nu}x^{-\nu-\frac{1}{2}}e^{-2x}\Big(1+\sum^{\infty}\limits_{j=1}\frac{c_{j}}{x^{j}}\Big),\ \ \ \ \ x\rightarrow +\infty,
\end{equation}
where  parameters $\theta_{0}$, $\theta_{\infty}$ satisfy the relation
$$
\theta_{0}=\theta_{\infty}-1=\nu.
$$
For $|\lambda|<1/\pi$,   the asymptotic behavior of $y(x;\nu,\lambda)$ as $x\rightarrow 0^{+}$ is  described by
$$ y(x;\nu,\lambda)\sim  B(2x)^{\sigma} $$
with the parameters $\sigma$ and $B$ are given as explicit functions of $\lambda$ and $\nu$, which are now known as the connection formulae.
When $\lambda>1/\pi$ and the parameter $\nu=0$, the asymptotic behavior of $y(x)$ as $x\rightarrow 0^{+}$ was also derived in \cite{MTW}
\begin{equation}\label{yP3Asy}
y(x;0,\lambda)= \frac{x}{2\mu}\sin\left\{2\mu\ln \frac{x}{4}-2\arg\Gamma(i\mu)\right\}+O(x^{3}),\ \ \ \ x\rightarrow 0^{+},
\end{equation}
where
$$\lambda=\frac{1}{\pi}\cosh(\pi\mu),\quad \mu>0.$$
For  $\lambda<-1/\pi$, the asymptotics of $y(x)$ as $x\rightarrow 0^{+}$ follows from \eqref{yP3Asy} and  the symmetry relation \cite[(4.127)]{MTW}
\begin{equation}\label{symP3}y(x; \nu,\lambda)=\frac{1}{y(x; \nu,-\lambda)}.\end{equation}
The asymptotic formula for general parameter $\nu\in\mathbb{R}$ was worked out in \cite{FFW}.
When  $\lambda>1/\pi$ and $\nu=0$, it is readily seen from \eqref{yP3Asy} that $y(x; 0, \lambda)$ admits a sequence of zeros $\{c_{n}\}_{n\in \mathbb{N}}$ lying on the positive real axis with $x=0$ being a limiting point:
\begin{equation}\label{zeroP3Asy}
c_{n}\sim 4\exp\left\{-\frac{n\pi}{\mu}+\frac{\arg\Gamma(i\mu)}{\mu}\right\} \to 0^{+}, \ \ \ n\rightarrow\infty.
\end{equation}
Moreover, it is straightforward to check that the Taylor expansions of $y$ at the zeros $\{c_{2n}\}_{n\in \mathbb{N}}$ and $\{c_{2n+1}\}_{n\in \mathbb{N}}$ are corresponding to $\sigma=\sigma_{+}$ and $\sigma=\sigma_{-}$ in \eqref{PIIIExpandAtzero}, respectively.  It is also seen from the relation \eqref{symP3} that there are infinitely many poles of $y(x; 0, \lambda)$ clustering at $x=0$ on the  positive real axis when $\lambda<-1/\pi$.

Let us consider the DHE equation \eqref{int:DHE} with  the accessory parameters $a$ and $q$ and the fixed parameters
$\gamma=1+\theta_0$ and $p=\frac{1}{4}(\theta_{\infty}+\theta_{0})$.
Applying Theorem \ref{thm:IsoSetDHE},
there is an isomonodromy set of pairs of accessory parameters $(a_n,q_n)$  such that  the DHE equation
corresponding to these parameters have the same monodromy data as the Painlev\'e III transcendent $y(x;0,\lambda)$ determined by the asymptotic behavior \eqref{yP3Asy}. Moreover, the parameter $a_n=c_{2n}$ and the  accessory parameters  $q_n$ are expressed in terms of the Laurent parameters of the PIII transcendents as given in \eqref{DHE1}.
Combining  \eqref{tauP3y}, \eqref{DHE1},  \eqref{yP3Asy} and \eqref{zeroP3Asy}, we obtain the asymptotic approximation of  the accessory parameters as stated in the following theorem.

{\thm\label{thm:AsyDHE}{ Let $(a_n,q_n)$ be the sequence of  accessory parameters  such that  the DHE equation \eqref{int:DHE} corresponding to these parameters has the same monodromy data as the Painlev\'e III transcendents
$y(x;0,\lambda)$ determined by the asymptotic behavior \eqref{yP3Asy}, then we have the asymptotic approximations as $n\to\infty$
\begin{equation}\label{thm:P3polean}
 a_{n}\sim 4\exp\left\{-\frac{2n\pi}{\mu}+\frac{\arg\Gamma(i\mu)}{\mu}\right\} \to 0^{+},  \end{equation}
and
\begin{equation}\label{thm:qP3Asy1}
q(a_{n})=-\frac{\mu^{2}+1}{4}+O(a_{n}^{2}).\end{equation}
Here the parameter
$\lambda=\frac{1}{\pi}\cosh(\pi\mu) $ and $ \mu>0 $.
}}

\begin{rem}
In Theorem \ref{thm:AsyDHE}, we have derived the asymptotics of  a isomonodromy sequence of accessory parameters  for  DHE  \eqref{int:DHE}. The asymptotics are expressed in terms of the parameters in the behavior  \eqref{yP3Asy} of the corresponding Painlev\'e III transcendents. It would be desirable to describe the asymptotics via the monodomy data as  given in Theorem \ref{thm:AsyRBHE} and Theorem \ref{thm:AsyCHE}.   However, to the best of our knowledge the connection  between the parameters  in the asymptotic behavior \eqref{yP3Asy} for the  Painlev\'e III transcendents  and the  monodomy data has not been worked out in the literature. We expect that such a connection formula could be derived, perhaps  by using the Riemann-Hilbert method or Isomonodromy method \cite{FIKN}. This, together with  Theorem \ref{thm:AsyDHE}, would then give us a description of the  asymptotics of the accessory parameters for DHE via the monodomy data.
We will leave this problem to a  future consideration.
\end{rem}

\section*{Acknowledgements}
The authors  are very grateful to the anonymous reviewers for their constructive comments and suggestions. 
The work of Shuai-Xia Xu was supported in part by the National Natural Science Foundation of China under grant   numbers 11571376 and 11971492. Yu-Qiu Zhao was supported in part by the National Natural Science Foundation of China under grant numbers    11571375 and 11971489.

\end{document}